\documentclass{article}

\usepackage{fullpage}
\usepackage{amsmath}
\usepackage{amsthm}
\usepackage{amssymb}
\usepackage{graphicx}
\usepackage{comment}
\usepackage{tikz}

\newtheorem{thm}{Theorem}[section]
\newtheorem{prop}[thm]{Proposition}
\newtheorem{cor}[thm]{Corollary}
\newtheorem{lemma}[thm]{Lemma}

\theoremstyle{definition}

\newtheorem*{ex}{Example}
\newtheorem*{rmk}{Remark}

\newcommand{\mo}[2]{[#1]_{#2}}

\title{Coefficients of a relative of cyclotomic polynomials}
\author{Ricky Ini Liu\\University of Michigan\\Ann Arbor, Michigan\\\texttt{riliu@umich.edu}}

\begin{document}
\maketitle

\begin{abstract}
Let $N=p_1p_2\cdots p_n$ be a product of $n$ distinct primes. Define $P_N(x)$ to be the polynomial $(1-x^N)\prod_{1\leq i<j\leq n}(1-x^{N/(p_ip_j)})/\prod_{i=1}^n (1-x^{N/p_i})$. (When $n=2$, $P_{pq}(x)$ is the $pq$-th cyclotomic polynomial, and when $n=3$, $P_{pqr}(x)$ is $(1-x)$ times the $pqr$-th cyclotomic polynomial.) Let the height of a polynomial be the maximum absolute value of one of its coefficients. It is well known that the height of $\Phi_{pq}(x)$ is 1, and Gallot and Moree \cite{GallotMoree} showed that the same is true for $P_{pqr}(x)$ when $n=3$. We show that the coefficients of $P_N(x)$ depend mainly on the relative order of sums of residues of the form $p_j^{-1} \pmod {p_i}$. This allows us to explicitly describe the coefficients of $P_N(x)$ when $n=3$ and show that the height of $P_N(x)$ is at most 2 when $n=4$. We also show that for any $n$ there exist $P_N(x)$ with height 1 but that in general the maximum height of $P_N(x)$ is a function depending only on $n$ with growth rate $2^{n^2/2+O(n\log n)}$.
\end{abstract}

\section{Introduction}

We say the \emph{height} of a polynomial in $\mathbf Z[x]$ is the maximum absolute value of any of its coefficients. It has long been known that the cyclotomic polynomials $\Phi_n(x)$ have a tendency to have small heights compared to $n$. For instance, if $n$ is divisible by at most two odd primes, then the height of $\Phi_n(x)$ is 1 \cite{LamLeung}. However, when $n = pqr$, where $p<q<r$ are odd primes, the height of $\Phi_n(x)$ is only bounded by a linear function in $p$ \cite{Beiter}. On the other hand, Gallot and Moree showed that any two adjacent coefficients of $\Phi_{pqr}(x)$ differ by at most 1, which is equivalent to showing that $(1-x)\Phi_{pqr}(x)$ has height 1 \cite{GallotMoree}. The heights of other cyclotomic polynomials and products of cyclotomic polynomials have been studied extensively elsewhere (see, for instance,  \cite{Elder}, \cite{GallotMoree2}, \cite{Kaplan}, \cite{Kaplan2}, \cite{Moree}, \cite{ZhaoZhang}).

We generalize $\Phi_{pq}(x)$ and $(1-x)\Phi_{pqr}(x)$ by considering the polynomial
\[P_N(x)=\frac{(1-x^N)\prod_{1\leq i<j\leq n}(1-x^{N_{ij}})}{\prod_{i=1}^n (1-x^{N_i})},\]
where $N=p_1p_2\dotsm p_n$ is a product of $n$ distinct primes and $N_{i_1\cdots i_m} = N/(p_{i_1}\cdots p_{i_m})$. In some sense, $P_N(x)$ is a toy version of the cyclotomic polynomial in that $\Phi_N(x)$ can be written as a rational function containing in the numerator or denominator all $(1-x^{N_{i_1\dots i_m}})$ while $P_N(x)$ contains only those for which $m \leq 2$.

Let $M(n)$ be the maximum height of $P_N(x)$ when $N$ has $n$ distinct prime factors. Since when $n=2$, $P_{pq}(x) = \Phi_{pq}(x)$, and when $n=3$, $P_{pqr}(x) = (1-x)\Phi_{pqr}(x)$, it is already known that $M(2)=M(3)=1$. We will show that $M(4)=2$ as well as that $M(n)$ exists for all $n$, so that the height of $P_N(x)$ is bounded by a function in $n$ that does not depend on the individual primes dividing $N$. We will also provide an explicit expression for the coefficients of $P_N(x)$ that generalizes the known expressions for the coefficients of $\Phi_{pq}(x)$, and we exhibit polynomials $P_N(x)$ with height 1 for any $n$. Finally, we will show that although $M(n)$ is small when $n$ is small, in fact $M(n)$ grows exponentially in $n^2$.

We begin in Section 2 with some preliminaries, including a short proof that the height of $\Phi_{pq}(x)$ is 1. In Section 3 we prove that the coefficients of $P_N(x)$ can be described in terms of the relative order of sums of residues of the form $p_j^{-1} \pmod {p_i}$, and we then use this to give an upper bound on $M(n)$. In Section 4 we explicitly and pictorially describe the coefficients of $P_N(x)$ for $n=3$ and prove that $M(4)=2$. In Section 5 we construct polynomials $P_N(x)$ with large height and thereby show that $M(n) = 2^{\frac 12n^2+O(n \log n)}$. Finally, in Section 6 we construct $P_N(x)$ with height 1 for all $n$.

\section{Preliminaries}
Let $n>1$ be a positive integer, and let $N=p_1p_2\dotsm p_n$ be the product of $n$ distinct primes. For ease of notation throughout,  we write $N_{i_1\cdots i_m} = N/(p_{i_1}\cdots p_{i_m})$ for any $i_1, \dots, i_m \in [n]$.

Our main object of study is the following:

\[P_N(x)=\frac{(1-x^N)\prod_{1\leq i<j\leq n}(1-x^{N_{ij}})}{\prod_{i=1}^n (1-x^{N_i})}.\]

\begin{rmk} Essentially all of the results below regarding the coefficients of $P_N(x)$ will hold even when the $p_i$ are not distinct primes but only pairwise relatively prime positive integers greater than 1. However, we will assume for the remainder that they are prime in order to simplify the statements of the results.
\end{rmk}

\begin{prop}
The rational function $P_N(x)$ is a polynomial with integer coefficients.
\end{prop}
\begin{proof}
Every root of the denominator is a root of unity with order of the form $p_{i_1}\dotsm p_{i_m}$. Such a root appears $n-m$ times in the denominator and $1+\binom{n-m}{2}$ times in the numerator. Since $\binom{n-m}{2}+1-(n-m)=\frac 12(n-m-1)(n-m-2) \geq 0$, $P_N(x)$ is a polynomial, and it has integer coefficients since both the numerator and denominator are monic integer polynomials (up to sign).
\end{proof}

When $n=2$ and $N=pq$, we have
\[P_{pq}(x) = \frac{(1-x^{pq})(1-x)}{(1-x^p)(1-x^q)} = \Phi_{pq}(x),\]
where $\Phi_{pq}$ is the $pq$-th cyclotomic polynomial. Likewise, when $n=3$ and $N=pqr$, we find that $P_{pqr}(x) = (1-x)\Phi_{pqr}(x)$.

It is well known that $\Phi_{pq}$ has all of its coefficients at most 1 in absolute value (see, for instance, \cite{LamLeung}), and Gallot and Moree \cite{GallotMoree} have recently shown that $(1-x)\Phi_{pqr}(x)$ also has this property. We give one proof of the result for $\Phi_{pq}$ below for illustrative purposes.

For a rational number $c=\frac{a}{b}$ with $b$ relatively prime to $m$, we will write $\mo cm$ for the smallest nonnegative integer $k$ such that $kb \equiv a \pmod m$. Note that $0 \leq \mo cm \leq m-1$.

We first need the following easy lemma.

\begin{lemma} \label{pq+1}
Let $p$ and $q$ be distinct primes. Then \[(1-x^{p\mo{p^{-1}}q}) + (x^{pq}-x^{q\mo{q^{-1}}p}) \equiv 1-x \pmod {(1-x^p)(1-x^q)}.\]

\end{lemma}
\begin{proof}
Note that $p\mo{p^{-1}}{q} + q \mo{q^{-1}}{p}$ is congruent to $1 \pmod p$ and $1 \pmod q$ and lies between 1 and $2pq$, so it must equal $pq+1$. Then the difference between the two sides is
\[x - x^{p\mo{p^{-1}}q} - x^{q\mo{q^{-1}}p} + x^{pq} = x(1-x^{p \mo{p^{-1}}q-1})(1-x^{q \mo{q^{-1}}p-1}),\]
and the two binomials are divisible by $1-x^q$ and $1-x^p$, respectively.
\end{proof}

We now derive an expression for $P_{pq}(x)$ that will allow us to easily extract its coefficients.

\begin{prop} \label{pq}
Modulo $1-x^{pq}$, $P_{pq}(x)$ is congruent to the polynomial
\[\frac{1-x^{pq}}{1-x^q}\cdot\frac{1-x^{p \mo {p^{-1}}{q}}}{1-x^p}+\frac{1-x^{pq}}{1-x^p}\cdot\frac{x^{pq}-x^{q \mo {q^{-1}}{p}}}{1-x^q}.\]
\end{prop}
\begin{proof}
By subtracting $P_{pq}(x)$ from the above expression and dividing by $1-x^{pq}$, we must show that
\[\frac{(1-x^{p \mo{p^{-1}}q})+(x^{pq}-x^{q \mo {q^{-1}}p}) - (1-x)}{(1-x^p)(1-x^q)}\]
is a polynomial. But this follows from Lemma~\ref{pq+1}.
\end{proof}

Let us write $\{a < b\}=1$  if $a<b$ and 0 otherwise (and similarly for other inequalities).

\begin{prop} \label{pq-coeffs}
For $0 \leq k < pq$, the coefficient of $x^k$ in $P_{pq}(x)$ is \[\{\mo{kp^{-1}}q<\mo{p^{-1}}q\}-\{\mo{kq^{-1}}p\geq \mo{q^{-1}}p\}.\]
\end{prop}
\begin{proof}
The first term in Proposition~\ref{pq} can be written as
\[(1+x^p+x^{2p}+\dots+x^{(\mo{p^{-1}}q-1)p})(1+x^q+x^{2q}+\dots+x^{(p-1)q}).\]
This has terms of the form $x^{ap+bq}$, where $0 \leq a < \mo{p^{-1}}q$ and $0 \leq b < p$. But $\mo{(ap+bq)p^{-1}}q = a$, so modulo $1-x^{pq}$ these are the terms $x^k$, $0 \leq k < pq$, such that $\mo{kp^{-1}}q < \mo{p^{-1}}q$. A similar analysis of the second term in Proposition~\ref{pq} (as well as the fact that $\deg P_{pq}<pq$) gives the result.
\end{proof}

One can describe Proposition~\ref{pq-coeffs} pictorially as follows. (A similar diagram can be found in \cite{Elder}.) Construct an array with $p$ rows and $q$ columns such that the entry in the $(i+1)$st row and $(j+1)$st column is $\mo{pj+qi}{pq}$. Then to find the coefficient of $x^k$, we add 1 if $k$ lies in the first $\mo{p^{-1}}q$ columns and subtract 1 if it lies in the last $p-\mo{q^{-1}}p$ rows. (See Figure~\ref{2d}.) There are therefore $\mo{p^{-1}}q\mo{q^{-1}}p$ coefficients equal to 1 and $(q-\mo{p^{-1}}q)(p-\mo{q^{-1}}p)=\mo{p^{-1}}q\mo{q^{-1}}p-1$ coefficients equal to $-1$.

\begin{figure}
\begin{center}
\begin{tabular}{c||ccc|cccc}
&+&+&+&&&&\\
\hline\hline
&0&5&10&15&20&25&30\\
&7&12&17&22&27&32&2\\
&14&19&24&29&34&4&9\\
\hline
--&21&26&31&1&6&11&16\\
--&28&33&3&8&13&18&23
\end{tabular}
\end{center}
\caption{Finding the coefficients of $P_{pq}(x)=\Phi_{pq}(x)$ when $p=5$ and $q=7$. Since $p\mo{p^{-1}}q + q\mo{q^{-1}}p \equiv 1 \pmod {pq}$, we draw lines directly to the left and above $1$ in the table. To find the coefficient of $x^k$, find $k$ and add the signs in the corresponding row and column. Therefore the coefficient is $1$ for the exponents in the northwest region, $-1$ for those in the southeast region, and $0$ for all others.}
\label{2d}
\end{figure}

Let the \emph{height} of a polynomial be the maximum absolute value of one of its coefficients. We will write $M(n)$ for the maximum height of $P_N(x)$ over all $N=p_1p_2\dotsm p_n$. We have just seen that $M(2)=1$.

The purpose of the next section is to provide a similar description of $P_N(x)$ for larger $n$ and thereby derive an upper bound on $M(n)$ (in particular showing that $M(n)$ exists).

\section{Coefficients of $P_N(x)$}

In Section~2, we have seen that for $0 \leq k < pq$, the coefficient of $x^k$ in $P_{pq}(x)$ depends only on the relative orders of the elements in $\{\mo{kp^{-1}}q, \mo{p^{-1}}q\}$ and $\{\mo{kq^{-1}}p, \mo{q^{-1}}p\}$. In this section we will show that a similar result holds for the coefficients of $P_N(x)$ in general. We will also provide an exponential upper bound on $M(n)$. This bound will be enough to show that $M(3)=1$, and we will see that this bound is in fact asymptotically tight in Section~5.

We first mimic the strategy of Proposition~\ref{pq}: we will write $P_N(x)$ modulo $1-x^N$ as a linear combination of $(1-x^N)/(1-x^{N_i})$. In fact, we will do so in $2^{\binom n2}$ different ways. (We remark that this approach essentially expresses $P_N(x)$ as a ``coboundary'' \`a la Question 25 of Musiker and Reiner \cite{MusikerReiner}.)

\begin{prop} \label{pn}
Let $S$ be a set such that for all integers $1 \leq i \neq j \leq n$, $S$ contains exactly one of the ordered pairs $(i,j)$ and $(j,i)$. Then modulo $1-x^N$, $P_N(x)$ is congruent to the polynomial

\[
\sum_{i=1}^n \biggl(
\frac{1-x^N}{1-x^{N_i}}\cdot
\prod_{(i,j) \in S} \frac{1-x^{\mo{p_i^{-1}}{p_j} \cdot N_j}}{1-x^{N_j}}\cdot
\prod_{(j,i) \in S} \frac{x^N-x^{\mo{p_i^{-1}}{p_j} \cdot N_j}}{1-x^{N_j}}\cdot
\prod_{\substack{1 \leq j_1 < j_2 \leq n\\j_1,j_2 \neq i}}\left(1-x^{N_{j_1j_2}}\right)
\biggr).
\]
\end{prop}
\begin{proof}
Note that the denominator of each term on the left is the same as the denominator of $P_N(x)$. Therefore it suffices to show that, modulo the denominator $\prod_{i=1}^n \left(1-x^{N_i}\right)$,
\[
\sum_{i=1}^n \biggl(
\prod_{(i,j) \in S}\left(1-x^{\mo{p_i^{-1}}{p_j}\cdot N_j}\right)\cdot
\prod_{(j,i) \in S}\left(x^N-x^{\mo{p_i^{-1}}{p_j}\cdot N_j}\right)\cdot
\prod_{\substack{1 \leq j_1 < j_2 \leq n\\j_1,j_2 \neq i}}\left(1-x^{N_{j_1j_2}}\right)
\biggr)
\equiv
\prod_{1 \leq j_1 < j_2 \leq n}\left(1-x^{N_{j_1j_2}}\right).
\]
Note that $\prod_{i=1}^n \left(1-x^{N_i}\right)$ is the least common multiple of $\left(1-x^{N_{i_1i_2\dots i_m}}\right)^m$ as $\{i_1, \dots, i_m\}$ ranges over nonempty subsets of $[n]$ because any root of order $N_{i_1\dots i_m}$ appears $m$ times. Therefore it suffices to check that the congruence holds modulo each $\left(1-x^{N_{i_1\dots i_m}}\right)^m$.

Note that $1-x^{N_{i_1\dots i_m}}$ divides the $i$th term on the left side $(m-1)+\binom{m-1}{2} = \binom{m}{2}$ times when $i \in \{i_1, \dots, i_m\}$ and $m+\binom{m}{2} = \binom{m+1}{2}$ times otherwise. It also divides the right side $\binom{m}{2}$ times.

When $m=1$, $1-x^{N_{i_1}}$ divides all the terms on the left except when $i=i_1$. Reducing the exponents in that term modulo $N_i$, it suffices to observe that \[\mo{p_i^{-1}}{p_j} \cdot N_j = N_i\cdot\frac{p_i\mo{p_i^{-1}}{p_j}-1}{p_j}+N_{ij} \equiv N_{ij} \pmod{N_i},\]
so each factor is congruent to the corresponding factor on the right.

When $m=2$, both sides are divisible by $1-x^{N_{i_1i_2}}$, and all but the $i_1$th and $i_2$th terms on the left are divisible by its square. Let $y = x^{N_{i_1i_2}}$, and without loss of generality, assume $(i_1, i_2) \in S$. In the $i_1$th term, the factor corresponding to $(i_1, i_2)$ is divisible by $1-y$, and any factor corresponding to $(i_1, j)$ or $(j, i_1)$ is congruent modulo $1-y$ to $1-x^{N_{i_1j}}$ by the calculation in the previous paragraph. Similarly reducing the $i_2$th term and ignoring the common factors on both sides, we find that it suffices to show that
\[(1-y^{p_{i_1}\mo{p_{i_1}^{-1}}{p_{i_2}}}) + (y^{p_{i_1}p_{i_2}} - y^{p_{i_2}\mo{p_{i_2}^{-1}}{p_{i_1}}}) \equiv 1-y \pmod{(1-y)^2}.\]
This follows from Lemma~\ref{pq+1}.

When $m \geq 3$, $\binom{m}{2} \geq m$, so both sides are divisible by $\left(1-x^{N_{i_1\dots i_m}}\right)^m$. This completes the proof.
\end{proof}

Using Proposition~\ref{pn}, we may now state a result similar to Proposition~\ref{pq-coeffs}. Let us write $\overline{P}_N(x)$ for the reduction of $P_N(x)$ modulo $1-x^N$ (so $\overline{P}_N(x)$ has degree less than $N$).
\begin{thm} \label{pn-coeffs}
Let $S$ be as in Proposition~\ref{pn}, and let 
\[
f_i(k)=\prod_{(i,j) \in S} \{\mo{kN_j^{-1}}{p_j} < \mo{p_i^{-1}}{p_j}\} \cdot
\prod_{(j,i) \in S}-\{\mo{kN_j^{-1}}{p_j} \geq \mo{p_i^{-1}}{p_j}\}.
\]
The coefficient of $x^k$ in $\overline{P}_N(x)$ is
\[
\sum_{i=1}^n \sum_{A \subset A_i} (-1)^{|A|}f_i(k-N_A),
\]
where $A_i$ is the set of all two-element subsets of $[n] \backslash \{i\}$ and $N_A = \sum_{\{j_i,j_2\} \in A} N_{j_1j_2}$. In particular, if $\deg P_N<N$, then this is also the coefficient of $x^k$ in $P_N(x)$.
\end{thm}
\begin{proof}
The expression
\[
\frac{1-x^N}{1-x^{N_i}}\cdot
\prod_{(i,j) \in S} \frac{1-x^{\mo{p_i^{-1}}{p_j} \cdot N_j}}{1-x^{N_j}}\cdot
\prod_{(j,i) \in S} \frac{x^N-x^{\mo{p_i^{-1}}{p_j} \cdot N_j}}{1-x^{N_j}}
\]
is the sum of terms $\pm x^k$, where $k=\sum a_jN_j$ such that $0 \leq a_i < p_i$, $0 \leq a_j < \mo{p_i^{-1}}{p_j}$ if $(i,j) \in S$, and $\mo{p_i^{-1}}{p_j} \leq a_j < p_j$ if $(j,i) \in S$, and the sign is given by the parity of the number of $j$ in this last case. Note that modulo $p_j$, $k \equiv a_jN_j$, so $a_j \equiv \mo{kN_j^{-1}}{p_j}$. Thus the above expression modulo $1-x^N$ is just $\sum_{k=0}^{N-1} f_i(k)$. The result now follows easily from Proposition~\ref{pn}.
\end{proof}

If $\overline{P}_N=P_N$, we find the following corollary.

\begin{cor} \label{order}
If $\deg P_N<N$, then the coefficient of $x^k$ in $P_N(x)$ for $k<N$ depends only on, for each $j$, the relative order of the $2^{n-1}+1$ residues $\mo{kN_j^{-1}}{p_j}$ and $\mo{\sum_{j' \in T} p_{j'}^{-1}}{p_j}$ for all $T \subset [n] \backslash\{j\}$.
\end{cor}
\begin{proof}
Observe that $f_i(k)$ depends only on the relative order of $\mo{kN_j^{-1}}{p_j}$, $0$, and $\mo{p_i^{-1}}{p_j}$ for each $j$. For any $A \subset A_i$ as in Theorem~\ref{pn-coeffs}, consider $k'=k-N_A$. We have $\mo{N_{j_1j_2}N_j^{-1}}{p_j}=0$ if $j \notin \{j_1, j_2\}$, while if instead $\{j_1, j_2\} = \{j,j'\}$, then $\mo{N_{jj'}N_j^{-1}}{p_j} = \mo{p_{j'}^{-1}}{p_j}$. Let $T_j=\{j' \mid \{j,j'\} \in A\}$. Then $k'N_j^{-1} \equiv kN_j^{-1}-\sum_{j' \in T_j}p_{j'}^{-1} \pmod{p_j}$. Thus $f_i(k')$ depends only on the relative order of $\mo{kN_j^{-1}}{p_j}$, $\mo{\sum_{j' \in T_j}p_{j'}^{-1}}{p_j}$, and $\mo{\sum_{j'\in T_j \cup \{i\}} p_{j'}^{-1}}{p_j}$ for each $j$. Considering all possible $A$ and $i$ and using Theorem~\ref{pn-coeffs} gives the result.
\end{proof}

In particular, these coefficients do not even depend on the specific primes $p_i$ as long as the order of the residues is given. We will say $N$ is \emph{generic} if $\deg P_N<N$ and all the $\mo{\sum_{j \in T} p_j^{-1}}{p_i}$ are distinct for any fixed $i$. In this case, it follows that if we plot the integers $0 \leq k <n$ at $(\mo{kN_1^{-1}}{p_1}, \dots, \mo{kN_n^{-1}}{p_n})$, there exists an analogous diagram to Figure~\ref{2d} with $n$ dimensions and $2^{n(n-1)}$ regions. We will investigate this diagram in the case $n=3$ in more detail in the next section.

The condition that $\deg P_N < N$ is fairly weak in that it holds always for small $n$ and ``most'' of the time for large $n$.

\begin{prop} \label{176}
If $\sum_{i=1}^n \frac{1}{p_i} < \frac{2n}{n-1}$, then $\deg P_N<N$. In particular, this holds if either $n<176$ or every prime $p_i$ is at least $\frac{n-1}{2}$.
\end{prop}
\begin{proof}
Let $A=\sum_{i=1}^n \frac{1}{p_i}$. By Maclaurin's inequality,
\[
1-\frac {\deg P_N}N = A-\sum_{1 \leq i < j \leq n} \frac{1}{p_ip_j}>A-\frac{n-1}{2n}\cdot A^2.
\]
Thus if $A< \frac{2n}{n-1}$, then $N>\deg P_N$.

If every prime $p_i$ is at least $\frac{n-1}{2}$, then clearly $A < \frac{2n}{n-1}$.

Suppose $\deg P_N\geq N$. We claim that $\deg P_{N'} \geq N'$, where $N'$ is the product of the first $n$ primes. It suffices to check that $A-\sum \frac{1}{p_ip_j}$ decreases if we reduce any $p_i$ and keep the others fixed. The coefficient of $\frac{1}{p_i}$ in this expression is $1-\sum_{j\neq i} \frac{1}{p_j} = 1+\frac{1}{p_i}-A$, which is negative since $A\geq \frac{2n}{n-1} > 2 > 1+\frac{1}{p_i}$, proving the claim. Therefore it suffices to check that whenever $N$ is the product of the first $n<176$ primes, $\deg P_N<N$. This follows from a straightforward computer calculation.
\end{proof}

The smallest $N$ for which $\deg P_N \geq N$ is the product of the first 176 primes, roughly $2.4182 \times 10^{439}$.


One can use Theorem~\ref{pn-coeffs} to give a bound on the largest absolute value of a coefficient of $P_N(x)$.

\begin{prop} \label{bound}
If $\deg P_N < N$, then every coefficient of $P_N(x)$ has absolute value at most $n\cdot 2^{\binom{n-2}{2}-1}$.
\end{prop}
\begin{proof}
Let us denote by $f_i^S$ the function defined in Theorem~\ref{pn-coeffs} corresponding to the set $S$. Then since the theorem holds for all $f_i^S$ (and the expression for the coefficients is linear in the $f_i^S$), it will also hold for the average $g_i = 2^{-\binom{n}{2}} \sum_S f_i^S$. Note that $|g_i(k)| = 2^{-(n-1)}$ for all $k$. Then
\[\left\lvert \sum_{i=1}^n \sum_{A \subset A_i} (-1)^{|A|}g_i(k-N_A) \right\rvert \leq
n \cdot 2^{|A_i|} \cdot 2^{-(n-1)} = n \cdot 2^{\binom{n-2}{2}-1}.\qedhere\]
\end{proof}

Although this bound looks rather weak, it is in fact fairly tight as we shall see in Section~5. As a special case, we can apply it when $n=3$ to recover the result of Gallot and Moree \cite{GallotMoree}.

\begin{cor}
$M(3)=1$. In other words, all of the coefficients of $(1-x)\Phi_{pqr}(x)$ have absolute value at most 1.
\end{cor}
\begin{proof}
Proposition~\ref{bound} (along with Proposition~\ref{176}) gives $M(3) \leq \frac 32$.
\end{proof}

Applying Proposition~\ref{bound} for $n=4$ gives $M(4) \leq 4$. We will show that in fact $M(4)=2$ in the next section.

We will need the following alternate descriptions of the coefficients of $P_N(x)$.

\begin{prop}\label{minus-n}
Let $p$ be a prime not dividing $N$, and let $a_N(k)$ be the coefficient of $x^k$ in $P_N(x)$. For any $T \subset [n]$, write $N_T=\sum_{i\in T}N_i$. Then
\[a_{pN}(k)-a_{pN}(k-N) = \sum (-1)^{|T|}a_N(p^{-1}(k-N_T)),\]
where the sum ranges over all $T \subset [n]$ for which $k \equiv N_T \pmod{p}$ (or, equivalently, for which $kN^{-1} \equiv \sum_{i \in T} \mo{p_i^{-1}}{p} \pmod{p}$).
\end{prop}
\begin{proof}
We can write
\[(1-x^N)P_{pN}(x) = P_N(x^p) \cdot \prod_{i=1}^n (1-x^{N_i}).\]
Computing the coefficient of $x^k$ on both sides gives the result.
\end{proof}

Note that if $N$ is generic, then this sum always has either 0 or 1 term.

\begin{prop} \label{general-coeffs}
Let $p$ be a prime not dividing $N$, and let $a_N(k)$ and $N_T$ be defined as in Proposition~\ref{minus-n}. If $m_k=p^{-1}(k-N\mo{kN^{-1}}{p})$, then
\[a_{pN}(k) - a_{pN}(k-pN) = \sum_{T \subset [n]} (-1)^{|T|} a_N(m_{k-N_T}).\]
\end{prop}
\begin{proof}
Expanding the left side as the telescoping sum \[(a_{pN}(k) - a_{pN}(k-N)) + (a_{pN}(k-N) - a_{pN}(k-2N)) + \dots + (a_{pN}(k-(p-1)N) - a_{pN}(k-pN))\]and  using Proposition~\ref{minus-n}, we see that each subset $T \subset [n]$ contributes to exactly one term $a_{pN}(k-cN)-a_{pN}(k-(c+1)N)$, where $k-cN \equiv N_T \pmod{p}$, or equivalently if $c=\mo{(k-N_T)N^{-1}}{p}$. Since $m_{k-N_T}=p^{-1}(k-N_T-cN)$, the result follows.
\end{proof}


We can use this to bound the growth of $M(n)$ without the restriction that $\deg P_N < N$.

\begin{prop}\label{general-bound}
$M(n) \leq 2^{\frac{1}{2} n^2 + O(n\log n)}$.
\end{prop}
\begin{proof}
Note that it suffices to check that $M(n)/M(n-1) \leq 2^{n+O(\log n)}$.

We may write $a_N(k) = \sum_{i=0}^{\infty}(a_N(k-iN)-a_N(k-(i+1)N)$. By Proposition~\ref{general-coeffs}, each term on the right side is bounded by $2^{n-1}M(n-1)$. Moreover, the number of nonzero terms on the right side is bounded by
\[\left\lceil \frac{\deg P_N}N \right\rceil \leq 1+\frac{\deg P_N}N < 1+ \sum_{i=1}^n\frac{1}{p_i}<n+1 = 2^{O(\log n)}.\]
Combining these gives the result.
\end{proof}
Note that although this bound is weak, it still grows like $2^{\frac12n^2}$ just like the bound found in Proposition~\ref{bound}.

In a certain special case, we can simplify Proposition~\ref{general-coeffs} to a form that will be useful later.

\begin{prop} \label{truncation}
Let $p$ be a prime not dividing $N=p_1\dotsm p_n$, and suppose that $\sum_{i=1}^n \frac{1}{p_i} < 1$. Then for any integer $k$,
\[a_{pN}(k) = \sum_{T \subset [n]}(-1)^{|T|}a_N(m_{k-N_T}') \cdot \{m_{k-N_T}' \leq kp^{-1}\},\]
where $m_k' = \mo{kp^{-1}}{N}$.
\end{prop}
\begin{proof}
By telescoping the sum in Proposition~\ref{general-coeffs}, we find that
\[
a_{pN}(k) = \sum_{T \subset [n]}(-1)^{|T|} \sum_{j=0}^\infty a_N(m_{k-N_T}-jN).
\]
By Proposition~\ref{176}, since $\sum_{i=1}^n \frac{1}{p_i} < 1$, $\deg P_N<N$. Thus each infinite sum has at most one nonzero term. Such a nonzero term must be $a_N(\mo{m_{k-N_T}}N) = a_N(m'_{k-N_T})$, and it is present if and only if $m'_{k-N_T} \leq m_{k-N_T}$. Since $m_{k-N_T}$ was constructed to be the largest integer $m$ such that $pm \equiv k-N_T \pmod N$ and $pm \leq k-N_T$, and $m'_{k-N_T}$ satisfies the first condition, we have that $m'_{k-N_T} \leq m_{k-N_T}$ if and only if $pm'_{k-N_T} \leq k-N_T$. But since both sides of this inequality are congruent modulo $N$, and $N_T= N\cdot\sum_{i \in T} \frac 1{p_i}<N$, we can equivalently drop the $N_T$ on the right side, and the result follows.
\end{proof}
In other words, to find $a_{pN}(k)$, we find a number of terms of the form $\pm a_N(m)$, order them by $m$, and sum some initial segment of them. Moreover, note that increasing $k$ by $N$ only changes the initial segment and not the terms. If $N$ is generic, then by Proposition~\ref{minus-n}, the sum over any initial segment will equal $a_{pN}(k')$ for some $k' \equiv k \pmod {N}$.

Another consequence of Proposition~\ref{minus-n} is that we can use it to show that many coefficients must vanish.

\begin{prop} \label{zero}
Suppose $\deg P_N < N$ and choose any $p_i$ dividing $N$. Let $0< k<N_i$ be an integer such that $kN_i^{-1}$ is not congruent to $\sum_{j \in T} \mo{p_j^{-1}}{p_i}$ modulo $p_i$ for any $T \subset [n]\backslash\{i\}$. Then $a_N(k)=0$.
\end{prop}
\begin{proof}
By Proposition~\ref{minus-n}, $a_N(k)= a_N(k-N_i)=0$ (since the sum on the righthand side is empty and $k-N_i<0$).
\end{proof}

Proposition~\ref{zero} implies that a large number of the regions defined by Corollary~\ref{order} must yield coefficients of zero. Conversely, if $a_N(k) \neq 0$ and $kN^{-1}$ is not congruent to any $\sum_{j \in T} \mo{p_j^{-1}}{p_i}$ modulo $p_i$, then $k>N_i$.

We will now give a more explicit (and visual) description of Corollary~\ref{order} for the case $n=3$ in the next section.

\section{The cases $n=3$ and $n=4$}

In this section, we will give a more explicit description of $P_N(x)$ when $n=3$. We will then use this to show that $M(4)=2$.

Let $N=pqr$. As per Corollary~\ref{order}, we first need to determine the possible orders of the residues in $\{0, \mo{p^{-1}}r, \mo{q^{-1}}r, \mo{p^{-1}+q^{-1}}r\}$ (and the analogous sets for the other primes). Let us write $\mo{c}m^+ = \mo{c}m$ if $\mo{c}m \neq 0$ and $\mo{c}m^+=m$ if $\mo{c}m = 0$.

\begin{lemma} \label{pqr}
Let $p$, $q$, and $r$ be distinct primes. Then, up to permutation of $p$, $q$, and $r$, one of the following four possibilities holds:
\begin{enumerate}
\item $\mo{p^{-1}+q^{-1}}r^+<\mo{p^{-1}}r\leq\mo{q^{-1}}r$,\\
$\mo{r^{-1}}q\leq\mo{p^{-1}}q<\mo{p^{-1}+r^{-1}}q^+$,\\
$\mo{q^{-1}}p\leq\mo{r^{-1}}p<\mo{q^{-1}+r^{-1}}p^+$;
\item $\mo{p^{-1}+q^{-1}}r^+<\mo{p^{-1}}r\leq\mo{q^{-1}}r$,\\
$\mo{r^{-1}}q\leq\mo{p^{-1}}q<\mo{p^{-1}+r^{-1}}q^+$,\\
$\mo{r^{-1}}p\leq\mo{q^{-1}}p<\mo{q^{-1}+r^{-1}}p^+$;
\item The same as (1) but with all inequalities reversed;
\item The same as (2) but with all inequalities reversed.
\end{enumerate}
\end{lemma}
\begin{proof}
Note that if $\frac{1}{r}(\mo{p^{-1}}r+\mo{q^{-1}}r)$ is at most 1, then $\mo{p^{-1}+q^{-1}}r^+$ is bigger than both $\mo{p^{-1}}r$ and $\mo{q^{-1}}r$, while if it is greater than 1, then these inequalities are reversed. Since $p\mo{p^{-1}}q+q\mo{q^{-1}}p = pq+1$, it follows that
\[\frac{1}{r}(\mo{p^{-1}}r+\mo{q^{-1}}r)+\frac{1}{q}(\mo{p^{-1}}q+\mo{r^{-1}}q)+\frac{1}{p}(\mo{q^{-1}}p+\mo{r^{-1}}p) = 3+\frac{1}{pq}+\frac{1}{pr}+\frac{1}{qr}. \label{eq:star} \tag{$*$}\]
Thus the three terms on the left side cannot all be at most than 1. Moreover, they cannot all be greater than 1, for then the left side would be at least $3+\frac 1p+\frac 1q+\frac 1r$. Thus either one or two of these terms is greater than 1.

Suppose that exactly one term on the left side of \eqref{eq:star} is greater than 1, say $\frac{1}{r}(\mo{p^{-1}}r+\mo{q^{-1}}r)$. We may assume without loss of generality that $\mo{p^{-1}}r\leq\mo{q^{-1}}r$. If $\mo{r^{-1}}q>\mo{p^{-1}}q$, then
\[
qr+1=q\mo{q^{-1}}r+r\mo{r^{-1}}q > q\mo{p^{-1}}r+r\mo{p^{-1}}q, \text{ so}\]
\[1\geq\frac 1r \mo{p^{-1}}r+\frac 1q \mo{p^{-1}}q = 2+\frac{1}{pr}+\frac{1}{pq}-\frac 1p (\mo{r^{-1}}p+ \mo{q^{-1}}p) \geq 1+\frac{1}{pr}+\frac{1}{pq}>1,\]
which is a contradiction. It follows that $\mo{r^{-1}}q\leq\mo{p^{-1}}q$, and the only possibilities are then (1) and (2).

Suppose two of the three terms on the left side of \eqref{eq:star} are greater than 1. We may assume without loss of generality that $\mo{q^{-1}}r \leq \mo{p^{-1}}r < \mo{p^{-1}+q^{-1}}r^+$. If $\mo{r^{-1}}q<\mo{p^{-1}}q$, then
\[
qr+1 = q\mo{q^{-1}}r+r\mo{r^{-1}}q < q\mo{p^{-1}}r+r\mo{p^{-1}}q, \text { so}\]
\[
1<\frac 1r \mo{p^{-1}}r+\frac 1q \mo{p^{-1}}q = 2+\frac 1{pr}+\frac 1{pq}-\frac 1p (\mo{r^{-1}}p+ \mo{q^{-1}}p) \leq 1+\frac 1{pr}+\frac{1}{pq}-\frac{1}{p}<1,
\]
which is a contradiction (since $\frac 1r+\frac 1q<1$). This yields either (3) and (4).
\end{proof}

It is now a simple matter to use Theorem~\ref{pn-coeffs} to calculate the coefficients of $P_{pqr}$. We will assume that $N$ is generic (all other $N$ can be obtained by degenerating these diagrams). For any integer $0 \leq k <pqr$, let \[h(k)=(\mo{k(qr)^{-1}}{p},\mo{k(pr)^{-1}}{q},\mo{k(pq)^{-1}}{r}).\]
By Theorem~\ref{pn-coeffs}, the coefficient of $x^k$ can be written as a sum of three terms depending on the projections of $h(k)$ onto the three coordinate planes. If $p$, $q$, and $r$ satisfy the first condition of Lemma~\ref{pqr} and the set $S$ in Theorem~\ref{pn-coeffs} is taken to be $\{(q,p), (r, p), (r, q)\}$, then we obtain Figure~\ref{case1}. The coefficient of $x^k$ is then the sum of the values corresponding to each of the three projections of $h(k)$ (where $+$ represents 1 and $-$ represents $-1$). This allows us to explicitly compute the coefficient of $x^k$ for each of the 64 regions as shown in Figure~\ref{case1}. (Note that each region is a product of three half-open intervals that contain their lower endpoints but not their upper endpoints.)

If $p$, $q$, and $r$ satisfy the second condition of Lemma~\ref{pqr}, the resulting coefficients are given in Figure~\ref{case2}. The coefficients for when the third or fourth condition is satisfied are given by these same two figures if we reverse the direction of each of the axes. Note that it is evident from this that all of the coefficients of $P_{pqr}$ are at most 1 in absolute value, so $M(3)=1$.

\begin{ex}
Let $p=5$, $q=11$, and $r=23$. Then
\begin{center}
\begin{tabular}{ccc}
$\mo{q^{-1}}{p}=1,$&$\mo{r^{-1}}{p}=2,$&$\mo{q^{-1}+r^{-1}}{p}=3,$\\
$\mo{r^{-1}}{q}=1,$&$\mo{p^{-1}}{q}=9,$&$\mo{p^{-1}+r^{-1}}{q}=10,$\\
$\mo{p^{-1}+q^{-1}}{r}=12,$&$\mo{p^{-1}}{r}=14,$&$\mo{q^{-1}}{r}=21,$
\end{tabular}
\end{center}
which is part of case 1 of Lemma~\ref{pqr}. To find the coefficient of $x^{71}$, we first calculate
\[h(71) = (\mo{71\cdot (11\cdot 23)^{-1}}{5}, \mo{71 \cdot (5 \cdot 23)^{-1}}{11}, \mo{71 \cdot (5 \cdot 11)^{-1}}{23}) = (2, 1, 13).\]
(Indeed, $2 \cdot (11 \cdot 23) + 1 \cdot (5 \cdot 23) + 13 \cdot (5 \cdot 11) = 1336 \equiv 71 \pmod{5\cdot 11\cdot 23}$.) Thus the coefficient corresponds to the region from $\mo{r^{-1}}{p}$ to $\mo{q^{-1}+r^{-1}}{p}$, from $\mo{r^{-1}}{q}$ to $\mo{p^{-1}}{q}$, and from $\mo{p^{-1}+q^{-1}}{r}$ to $\mo{p^{-1}}{r}$. Then Figure~\ref{case1} shows that the coefficient of $x^{71}$ equals 1 (as indicated by the box).
\end{ex}

\begin{figure}
\begin{center}
\begin{tikzpicture}[scale=0.3,z={(0,7)},y={(3.2,-.75)}, x={(-6.4,1)}]
\tikzstyle{every node}=[font=\small]

\draw (0.5, 0.5, .5) node{$1$};
\draw (1.5, 0.5, .5) node{$1$};
\draw (2.5, 0.5, .5) node{$0$};
\draw (3.5, 0.5, .5) node{$0$};
\draw (0.5, 1.5, .5) node{$0$};
\draw (1.5, 1.5, .5) node{$0$};
\draw (2.5, 1.5, .5) node{$0$};
\draw (3.5, 1.5, .5) node{$0$};
\draw (0.5, 2.5, .5) node{$0$};
\draw (1.5, 2.5, .5) node{$-1$};
\draw (2.5, 2.5, .5) node{$-1$};
\draw (3.5, 2.5, .5) node{$0$};
\draw (0.5, 3.5, .5) node{$0$};
\draw (1.5, 3.5, .5) node{$0$};
\draw (2.5, 3.5, .5) node{$0$};
\draw (3.5, 3.5, .5) node{$0$};

\draw (0.5, 0.5, 1.5) node{$0$};
\draw (1.5, 0.5, 1.5) node{$0$};
\draw (2.5, 0.5, 1.5) node{$0$};
\draw (3.5, 0.5, 1.5) node{$-1$};
\draw (0.5, 1.5, 1.5) node{$0$};
\draw (1.5, 1.5, 1.5) node{$0$};
\draw (2.5, 1.5, 1.5) node[draw, rectangle]{$1$};
\draw (3.5, 1.5, 1.5) node{$0$};
\draw (0.5, 2.5, 1.5) node{$0$};
\draw (1.5, 2.5, 1.5) node{$-1$};
\draw (2.5, 2.5, 1.5) node{$0$};
\draw (3.5, 2.5, 1.5) node{$0$};
\draw (0.5, 3.5, 1.5) node{$-1$};
\draw (1.5, 3.5, 1.5) node{$-1$};
\draw (2.5, 3.5, 1.5) node{$0$};
\draw (3.5, 3.5, 1.5) node{$-1$};

\draw (0.5, 0.5, 2.5) node{$0$};
\draw (1.5, 0.5, 2.5) node{$0$};
\draw (2.5, 0.5, 2.5) node{$-1$};
\draw (3.5, 0.5, 2.5) node{$-1$};
\draw (0.5, 1.5, 2.5) node{$0$};
\draw (1.5, 1.5, 2.5) node{$0$};
\draw (2.5, 1.5, 2.5) node{$0$};
\draw (3.5, 1.5, 2.5) node{$0$};
\draw (0.5, 2.5, 2.5) node{$1$};
\draw (1.5, 2.5, 2.5) node{$0$};
\draw (2.5, 2.5, 2.5) node{$0$};
\draw (3.5, 2.5, 2.5) node{$1$};
\draw (0.5, 3.5, 2.5) node{$0$};
\draw (1.5, 3.5, 2.5) node{$0$};
\draw (2.5, 3.5, 2.5) node{$0$};
\draw (3.5, 3.5, 2.5) node{$0$};

\draw (0.5, 0.5, 3.5) node{$0$};
\draw (1.5, 0.5, 3.5) node{$1$};
\draw (2.5, 0.5, 3.5) node{$0$};
\draw (3.5, 0.5, 3.5) node{$0$};
\draw (0.5, 1.5, 3.5) node{$-1$};
\draw (1.5, 1.5, 3.5) node{$0$};
\draw (2.5, 1.5, 3.5) node{$0$};
\draw (3.5, 1.5, 3.5) node{$0$};
\draw (0.5, 2.5, 3.5) node{$0$};
\draw (1.5, 2.5, 3.5) node{$0$};
\draw (2.5, 2.5, 3.5) node{$0$};
\draw (3.5, 2.5, 3.5) node{$1$};
\draw (0.5, 3.5, 3.5) node{$0$};
\draw (1.5, 3.5, 3.5) node{$1$};
\draw (2.5, 3.5, 3.5) node{$1$};
\draw (3.5, 3.5, 3.5) node{$1$}; 

\foreach \x in {0, ..., 3}
{
	\draw (\x, 0, -.75) -- (\x, 4, -.75);
	\draw (0, \x, -.75) -- (4, \x, -.75);
	\draw (\x, -4.25, 0) -- (\x, -4.25, 4);
	\draw (0, -4.25, \x) -- (4, -4.25, \x); 
	\draw (-4.25, \x, 0) -- (-4.25, \x, 4);
	\draw (-4.25, 0, \x) -- (-4.25, 4, \x);
}
\draw[dashed] (4, 0, -.75) -- (4, 4, -.75);
\draw[dashed] (0, 4, -.75) -- (4, 4, -.75);
\draw[dashed] (4, -4.25, 0) -- (4, -4.25, 4);
\draw[dashed] (0, -4.25, 4) -- (4, -4.25, 4); 
\draw[dashed] (-4.25, 4, 0) -- (-4.25, 4, 4);
\draw[dashed] (-4.25, 0, 4) -- (-4.25, 4, 4);

\draw (0, 0, -.75) node[above]{$0$};
\draw (1, 0, -.75) node[above left]{$\mo{q^{-1}}{p}$};
\draw (2, 0, -.75) node[above left]{$\mo{r^{-1}}{p}$};
\draw (3, 0, -.75) node[above left]{$\mo{q^{-1}+r^{-1}}{p}$};
\draw (4, 0, -.75) node[left]{$p$};
\draw (0, 1, -.75) node[above right]{$\mo{r^{-1}}{q}$};
\draw (0, 2, -.75) node[above right]{$\mo{p^{-1}}{q}$};
\draw (0, 3, -.75) node[above right]{$\mo{p^{-1}+r^{-1}}{q}$};
\draw (0, 4, -.75) node[right]{$q$};

\draw (0, -4.25, 0) node[below right]{$0$};
\draw (1, -4.25, 0) node[below]{$\mo{q^{-1}}{p}$};
\draw (2, -4.25, 0) node[below]{$\mo{r^{-1}}{p}$};
\draw (3, -4.25, 0) node[below=4pt]{$\mo{q^{-1}+r^{-1}}{p}$};
\draw (4, -4.25, 0) node[below left]{$p$};
\draw (0, -4.25, 1) node[right]{$\mo{p^{-1}+q^{-1}}{r}$};
\draw (0, -4.25, 2) node[right]{$\mo{p^{-1}}{r}$};
\draw (0, -4.25, 3) node[right]{$\mo{q^{-1}}{r}$};
\draw (0, -4.25, 4) node[right]{$r$};

\draw (-4.25, 0, 0) node[below left]{$0$};
\draw (-4.25, 1, 0) node[below]{$\mo{r^{-1}}{q}$};
\draw (-4.25, 2, 0) node[below]{$\mo{p^{-1}}{q}$};
\draw (-4.25, 3, 0) node[below=4pt]{$\mo{p^{-1}+r^{-1}}{q}$};
\draw (-4.25, 4, 0) node[below right]{$q$};
\draw (-4.25, 0, 1) node[left]{$\mo{p^{-1}+q^{-1}}{r}$};
\draw (-4.25, 0, 2) node[left]{$\mo{p^{-1}}{r}$};
\draw (-4.25, 0, 3) node[left]{$\mo{q^{-1}}{r}$};
\draw (-4.25, 0, 4) node[left]{$r$};

\draw (.5, .5, -.75) node{$+$};
\draw (1.5, .5, -.75) node{$+$};
\draw (1.5, 2.5, -.75) node{$-$};
\draw (2.5, 2.5, -.75) node{$-$};

\draw (0.5, -4.25, 3.5) node{$-$};
\draw (2.5, -4.25, 1.5) node{$+$};

\draw (-4.25, .5, 1.5) node{$-$};
\draw (-4.25, .5, 2.5) node{$-$};
\draw (-4.25, 3.5, 1.5) node{$-$};
\draw (-4.25, 2.5, 2.5) node{$+$};
\draw (-4.25, 2.5, 3.5) node{$+$};
\draw (-4.25, 3.5, 3.5) node{$+$};

\end{tikzpicture}
\end{center}
\caption{Coefficients of $P_{pqr}(x)$ in case 1 of Lemma~\ref{pqr}. To find the coefficient of $x^k$, compute $h(k)=(\mo{k(qr)^{-1}}{p},\mo{k(pr)^{-1}}{q},\mo{k(pq)^{-1}}{r})$ and determine which of the 64 regions it lies in. (All intervals contain their lower endpoints but not their upper endpoints.) According to Theorem~\ref{pn-coeffs}, the coefficient of $x^k$ can be determined by summing the $+$'s and $-$'s from its three projections of its region as indicated.} For instance, the coefficient of $x^{71}$ in $P_{5 \cdot 11 \cdot 23}(x)$ is the boxed 1, which received a contribution of $+1$ from the projection along the $q$ direction and no contributions from the other two projections..
\label{case1}
\end{figure}

\begin{figure}
\begin{center}
\begin{tikzpicture}[scale=0.3,z={(0,7)},y={(3.2,-.75)}, x={(-6.4,1)}]
\tikzstyle{every node}=[font=\small]

\draw (0.5, 0.5, .5) node{$1$};
\draw (1.5, 0.5, .5) node{$0$};
\draw (2.5, 0.5, .5) node{$0$};
\draw (3.5, 0.5, .5) node{$0$};
\draw (0.5, 1.5, .5) node{$0$};
\draw (1.5, 1.5, .5) node{$0$};
\draw (2.5, 1.5, .5) node{$0$};
\draw (3.5, 1.5, .5) node{$0$};
\draw (0.5, 2.5, .5) node{$0$};
\draw (1.5, 2.5, .5) node{$0$};
\draw (2.5, 2.5, .5) node{$-1$};
\draw (3.5, 2.5, .5) node{$0$};
\draw (0.5, 3.5, .5) node{$0$};
\draw (1.5, 3.5, .5) node{$0$};
\draw (2.5, 3.5, .5) node{$0$};
\draw (3.5, 3.5, .5) node{$0$};

\draw (0.5, 0.5, 1.5) node{$0$};
\draw (1.5, 0.5, 1.5) node{$0$};
\draw (2.5, 0.5, 1.5) node{$0$};
\draw (3.5, 0.5, 1.5) node{$-1$};
\draw (0.5, 1.5, 1.5) node{$0$};
\draw (1.5, 1.5, 1.5) node{$1$};
\draw (2.5, 1.5, 1.5) node{$1$};
\draw (3.5, 1.5, 1.5) node{$0$};
\draw (0.5, 2.5, 1.5) node{$0$};
\draw (1.5, 2.5, 1.5) node{$1$};
\draw (2.5, 2.5, 1.5) node{$0$};
\draw (3.5, 2.5, 1.5) node{$0$};
\draw (0.5, 3.5, 1.5) node{$-1$};
\draw (1.5, 3.5, 1.5) node{$0$};
\draw (2.5, 3.5, 1.5) node{$0$};
\draw (3.5, 3.5, 1.5) node{$-1$};

\draw (0.5, 0.5, 2.5) node{$0$};
\draw (1.5, 0.5, 2.5) node{$-1$};
\draw (2.5, 0.5, 2.5) node{$-1$};
\draw (3.5, 0.5, 2.5) node{$-1$};
\draw (0.5, 1.5, 2.5) node{$0$};
\draw (1.5, 1.5, 2.5) node{$0$};
\draw (2.5, 1.5, 2.5) node{$0$};
\draw (3.5, 1.5, 2.5) node{$0$};
\draw (0.5, 2.5, 2.5) node{$1$};
\draw (1.5, 2.5, 2.5) node{$1$};
\draw (2.5, 2.5, 2.5) node{$0$};
\draw (3.5, 2.5, 2.5) node{$1$};
\draw (0.5, 3.5, 2.5) node{$0$};
\draw (1.5, 3.5, 2.5) node{$0$};
\draw (2.5, 3.5, 2.5) node{$0$};
\draw (3.5, 3.5, 2.5) node{$0$};

\draw (0.5, 0.5, 3.5) node{$0$};
\draw (1.5, 0.5, 3.5) node{$-1$};
\draw (2.5, 0.5, 3.5) node{$0$};
\draw (3.5, 0.5, 3.5) node{$0$};
\draw (0.5, 1.5, 3.5) node{$-1$};
\draw (1.5, 1.5, 3.5) node{$-1$};
\draw (2.5, 1.5, 3.5) node{$0$};
\draw (3.5, 1.5, 3.5) node{$0$};
\draw (0.5, 2.5, 3.5) node{$0$};
\draw (1.5, 2.5, 3.5) node{$0$};
\draw (2.5, 2.5, 3.5) node{$0$};
\draw (3.5, 2.5, 3.5) node{$1$};
\draw (0.5, 3.5, 3.5) node{$0$};
\draw (1.5, 3.5, 3.5) node{$0$};
\draw (2.5, 3.5, 3.5) node{$1$};
\draw (3.5, 3.5, 3.5) node{$1$}; 

\foreach \x in {0, ..., 3}
{
	\draw (\x, 0, -.75) -- (\x, 4, -.75);
	\draw (0, \x, -.75) -- (4, \x, -.75);
	\draw (\x, -4.25, 0) -- (\x, -4.25, 4);
	\draw (0, -4.25, \x) -- (4, -4.25, \x); 
	\draw (-4.25, \x, 0) -- (-4.25, \x, 4);
	\draw (-4.25, 0, \x) -- (-4.25, 4, \x);
}
\draw[dashed] (4, 0, -.75) -- (4, 4, -.75);
\draw[dashed] (0, 4, -.75) -- (4, 4, -.75);
\draw[dashed] (4, -4.25, 0) -- (4, -4.25, 4);
\draw[dashed] (0, -4.25, 4) -- (4, -4.25, 4); 
\draw[dashed] (-4.25, 4, 0) -- (-4.25, 4, 4);
\draw[dashed] (-4.25, 0, 4) -- (-4.25, 4, 4);

\draw (0, 0, -.75) node[above]{$0$};
\draw (1, 0, -.75) node[above left]{$\mo{r^{-1}}{p}$};
\draw (2, 0, -.75) node[above left]{$\mo{q^{-1}}{p}$};
\draw (3, 0, -.75) node[above left]{$\mo{q^{-1}+r^{-1}}{p}$};
\draw (4, 0, -.75) node[left]{$p$};
\draw (0, 1, -.75) node[above right]{$\mo{r^{-1}}{q}$};
\draw (0, 2, -.75) node[above right]{$\mo{p^{-1}}{q}$};
\draw (0, 3, -.75) node[above right]{$\mo{p^{-1}+r^{-1}}{q}$};
\draw (0, 4, -.75) node[right]{$q$};

\draw (0, -4.25, 0) node[below right]{$0$};
\draw (1, -4.25, 0) node[below]{$\mo{r^{-1}}{p}$};
\draw (2, -4.25, 0) node[below]{$\mo{q^{-1}}{p}$};
\draw (3, -4.25, 0) node[below=4pt]{$\mo{q^{-1}+r^{-1}}{p}$};
\draw (4, -4.25, 0) node[below left]{$p$};
\draw (0, -4.25, 1) node[right]{$\mo{p^{-1}+q^{-1}}{r}$};
\draw (0, -4.25, 2) node[right]{$\mo{p^{-1}}{r}$};
\draw (0, -4.25, 3) node[right]{$\mo{q^{-1}}{r}$};
\draw (0, -4.25, 4) node[right]{$r$};

\draw (-4.25, 0, 0) node[below left]{$0$};
\draw (-4.25, 1, 0) node[below]{$\mo{r^{-1}}{q}$};
\draw (-4.25, 2, 0) node[below]{$\mo{p^{-1}}{q}$};
\draw (-4.25, 3, 0) node[below=4pt]{$\mo{p^{-1}+r^{-1}}{q}$};
\draw (-4.25, 4, 0) node[below right]{$q$};
\draw (-4.25, 0, 1) node[left]{$\mo{p^{-1}+q^{-1}}{r}$};
\draw (-4.25, 0, 2) node[left]{$\mo{p^{-1}}{r}$};
\draw (-4.25, 0, 3) node[left]{$\mo{q^{-1}}{r}$};
\draw (-4.25, 0, 4) node[left]{$r$};

\draw (.5, .5, -.75) node{$+$};
\draw (2.5, 2.5, -.75) node{$-$};

\draw (0.5, -4.25, 3.5) node{$-$};
\draw (1.5, -4.25, 3.5) node{$-$};
\draw (1.5, -4.25, 1.5) node{$+$}; 
\draw (2.5, -4.25, 1.5) node{$+$};

\draw (-4.25, .5, 1.5) node{$-$};
\draw (-4.25, .5, 2.5) node{$-$};
\draw (-4.25, 3.5, 1.5) node{$-$};
\draw (-4.25, 2.5, 2.5) node{$+$};
\draw (-4.25, 2.5, 3.5) node{$+$};
\draw (-4.25, 3.5, 3.5) node{$+$};

\end{tikzpicture}
\end{center}
\caption{Coefficients of $P_{pqr}(x)$ in case 2 of Lemma~\ref{pqr}.}
\label{case2}
\end{figure}
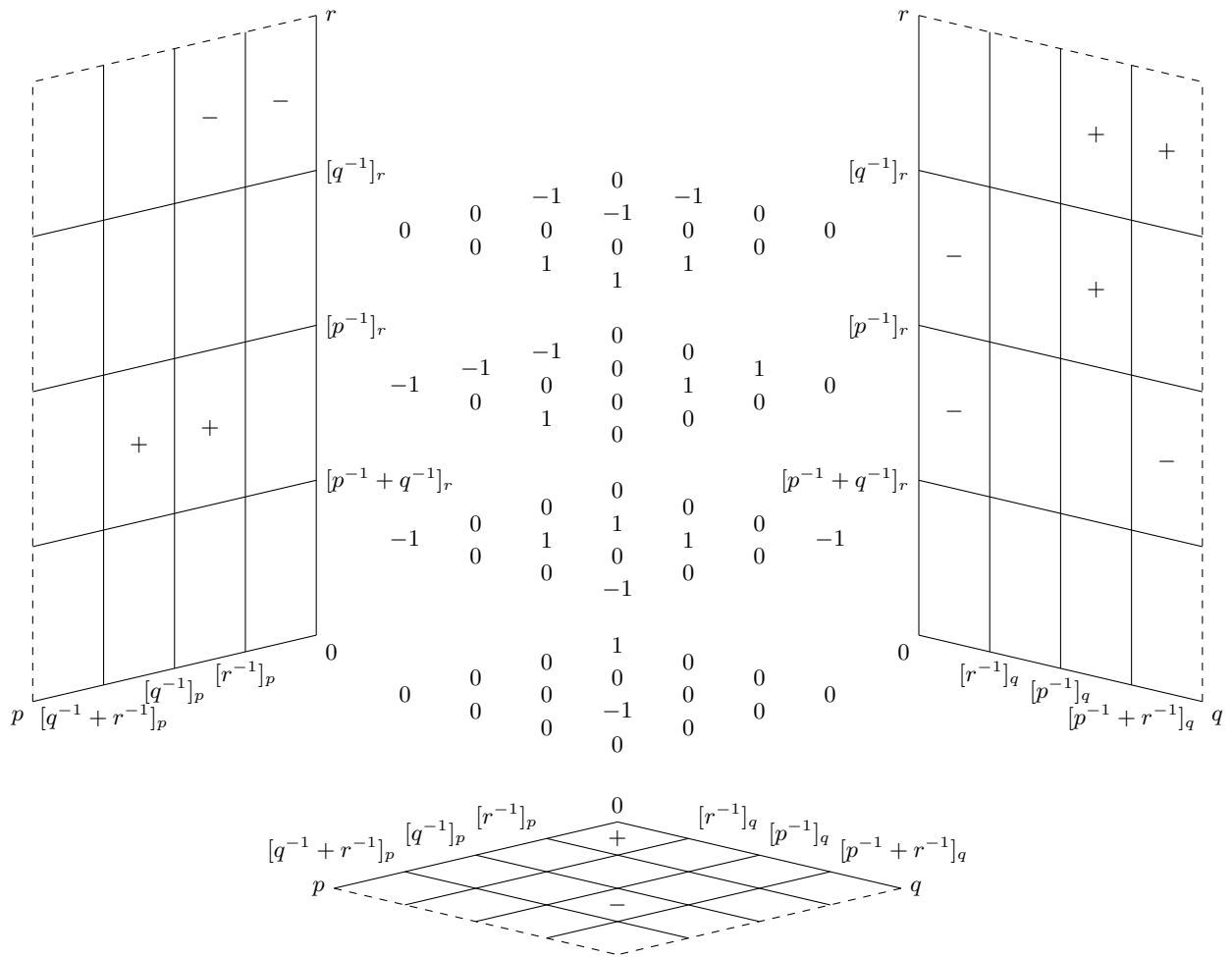



We can also use these diagrams showing the coefficients of $P_{pqr}(x)$ to prove that $M(4) =2$.

\begin{thm}
$M(4)=2.$
\end{thm}
\begin{proof}
Let $p$, $q$, $r$, and $s$ be arbitrary distinct primes. Also let us assume that $pqr$ satisfies the first condition of Lemma~\ref{pqr}; the other cases are similar. For convenience, let us denote each of the 64 regions in Figure~\ref{case1} by a triple of integers from $0$ to $3$ denoting its position in the $p$, $q$, and $r$ directions, and let $f(xyz)$ denote the coefficient corresponding to region $xyz$. For instance, region $211$ will refer to the region containing the coefficient of $x^{71}$ in the example above, and so $f(211)=1$.

Fix a residue $\bar{k}$ modulo $N=pqr$. By Proposition~\ref{truncation}, the coefficients of $x^k$ for $k \equiv \bar{k} \pmod{N}$ in $P_{Ns}(x)$ are partial sums of a signed sequence of eight coefficients of $P_{N}(x)$ that when plotted into the regions of Figure~\ref{case1} lie at the corners of a rectangular box. Let us denote the corresponding exponents by $k_{000}, k_{001}, k_{010}, \dots, k_{111}$ and suppose that they lie in the eight (not necessarily distinct) regions $x_0y_0z_0,\dots, x_1y_1z_1$, so the corresponding signed term is $g_{abc}=(-1)^{a+b+c} f(x_ay_bz_c)$. The order in which the $g_{abc}$ are summed is given by the order of the $k_{abc}$. Note that switching $k_{0bc}$ and $k_{1bc}$ for all $b$ and $c$ will only swap $x_0$ and $x_1$ and will therefore just change the sign of all the $g_{abc}$ and not the order in which we sum them. Since this will not affect the absolute value of any partial sum, we will assume that $h(k_{000})$ is minimal in all three coordinate directions.

Note that the sum of all eight of the $g_{abc}$ is 0 since the coefficients of $x^k$ vanish for $k$ sufficiently large. Then in order for a partial sum of these terms to be at least 3 in absolute value, at least six of the $g_{abc}$ must be nonzero.

First suppose all the $x_ay_bz_c$ are distinct. Then by examining Figure~\ref{case1} we see that there are only two places this can occur, namely either
\[x_ay_bz_c \in \{1,3\} \times \{0,3\} \times \{1,3\} \quad\text{or}\quad x_ay_bz_c \in \{1,3\} \times \{2,3\} \times \{1,3\}.\]
In the first case, we find
\[(g_{000}, g_{001}, g_{010}, g_{011}, g_{100}, g_{101}, g_{110}, g_{111}) = (0, -1, -1, 1, 1, 0, -1, 1).\]
In order to have a coefficient of absolute value at least 3, we must have that when the $g_{abc}$ are ordered according to the $k_{abc}$, all the 1's must come before all the $-1$'s or vice versa. But this cannot happen: by Proposition~\ref{zero}, $k_{011}<k_{001}$ (since $f(133)$ and $f(103)$ are both nonzero) and likewise $k_{110}<k_{100}$. The second case is similar.

Therefore we may suppose that not all of the $x_ay_bz_c$ are distinct, say $x_0=x_1$ (the other cases are similar). Then by Proposition~\ref{zero}, $k_{0bc} \leq k_{1bc}$ if $g_{0bc}=-g_{1bc}$ is nonzero. Then in order for some partial sum of the $g_{abc}$ to have absolute value at least 3, at least three of the $g_{0bc}$ must have the same sign. Then some $f(x_0y_bz_c)$ must differ in sign from both $f(x_0y_{1-b}z_c)$ and $f(x_0y_bz_{1-c})$. But an inspection of Figure~\ref{case1} shows that this cannot happen. 

We have shown that $M(4)\leq 2$. Since the coefficient of $x^{233}$ in $P_{5\cdot7\cdot11\cdot13}(x)$ is $-2$, we must have $M(4)=2$, as desired.
\end{proof}

Having analyzed $M(n)$ for small values of $n$, we will now give an asymptotic bound for $M(n)$ in the next section.

\section{Asymptotics of $M(n)$}

In this section, we will show that, although from the values $M(2)=M(3)=1$ and $M(4)=2$ it might appear that $M(n)$ grows slowly, in fact it grows exponentially in $n^2$. To place a lower bound on $M(n)$, we will inductively construct a polynomial $P_N(x)$ with a large coefficient by applying Proposition~\ref{truncation}.

We first need to show that there exist $N$ which yield arbitrarily large regions. Let $N=p_1\dotsm p_n$ be generic, and let $S_j(N) = \{\mo{\sum_{i \in T} p_{i}^{-1}}{p_j} \mid T \subset [n]\backslash\{j\}\}$. Let us write $d(S_j(N))$ for the smallest difference between two elements of $S_j(N) \cup \{p_j\}$ (so that it is the smallest length of any region in the $j$th direction). We will examine what happens when we add to $p_j$ a multiple of $N_j$. Clearly this will not change $S_i(N)$ for any $i \neq j$.

\begin{lemma}\label{z}
Fix $T \subset [n]\backslash\{j\}$, and define $z_T$ to be the fractional part of $\sum_{i \in T}(1-\frac{1}{p_i}\mo{p_j^{-1}}{p_i})$. Then $p_j\cdot z_T + \sum_{i \in T} \frac 1{p_i}$ is an integer congruent to $\sum_{i \in T}p_i^{-1}$ modulo $p_j$. In particular, if $\sum_{i\in T} \frac{1}{p_i} < 1$, then $\mo{\sum_{i \in T} p_i^{-1}}{p_j}^+ = \lceil p_j \cdot z_T \rceil$.
\end{lemma}
\begin{proof}
Recall that $p_ip_j+1=p_i\mo{p_i^{-1}}{p_j}+p_j\mo{p_j^{-1}}{p_i}$, so that \[\mo{p_i^{-1}}{p_j} = p_j\left(1-\frac 1{p_i}\mo{p_j^{-1}}{p_i}\right)+\frac 1{p_i}.\]
Summing over all $i$ gives the result.
\end{proof}
Note that $z_T$ only depends on the residue of $p_j$ modulo $N_j$. This means that increasing $p_j$ by some multiple of $N_j$ will tend to keep the residues we care about in the same order while increasing the gaps in between them.

\begin{lemma}\label{enlarge}
Suppose that $\sum_{i=1}^n \frac{1}{p_i}<1$ and $N$ is generic. Then there exists $N'=p_1'p_2' \cdots p_n'$ and a positive constant $c<1$ such that for all $j$, $p_j'>p_j$, the corresponding elements of $S_j(N)$ and $S_j(N')$ are in the same order, and $\lfloor\frac{n}{2}\rfloor+1<cp_j'<d(S_j(N'))$.
\end{lemma}
\begin{proof}
We first construct $N''=p_1''p_2'' \cdots p_n''$ such that for all $j$, $p_j''>p_j$, the corresponding elements of $S_j(N)$ and $S_j(N'')$ are in the same order, and $d(S_j(N'')) \geq 3$.

Fix some $j$. By Lemma~\ref{z}, the elements of $S_j(N)$ (which are distinct since $N$ is generic) are in the same order as the $z_T$. Moreover, if we add a multiple of $N_j$ to any $p_j$ to get $p''_j$, this does not change the values (or order) of the $z_T$. By choosing $p''_j$ sufficiently large, we can ensure that the difference between any two $p''_j \cdot z_T$ is at least 3 (or as large as needed). Applying this process for all $j$ gives the result. (In order to ensure that the $p''_j$ are prime, we use the fact that there are infinitely many primes congruent to $p_j$ modulo $N_j$.)

Now choose $c$ to be smaller than each $\frac{1}{p_j''}$. Then, as above, replace each $p_j''$ with some $p_j'>\frac{\lfloor n/2\rfloor+1}{c}$ by adding multiples of the products of the other primes. We claim that the resulting $N'$ satisfies the desired properties. By the argument above, we need only check the desired inequalities. Therefore it suffices to show that $\frac{1}{p_j''} < \frac{d(S_j(N'))}{p_j'}$. Fix $j$, and suppose that we are replacing $p_j''$ with $p_j'$. Let the smallest difference $d(S_j(N'))$ between two elements of $S_j(N') \cup \{p_j'\}$ occur between $\lceil p_j' \cdot z_T \rceil$ and $\lceil p_j' \cdot z_{T'}\rceil$. (If the difference involves $p_j'$, take $1$ for the corresponding $z_T$.) Then $\frac{d(S_j(N'))+1}{p_j'} > |z_T-z_{T'}|$. But we know from before we replaced $p_j''$ that $|\lceil p_j'' \cdot z_T \rceil - \lceil p_j'' \cdot z_{T'} \rceil | \geq d(S_j(N''))$, so
$|z_T-z_{T'}| > \frac{d(S_j(N''))-1}{p_j''}.$ We find that
\[\frac{d(S_j(N'))}{p_j'} > |z_T-z_{T'}| - \frac{1}{p_j'} > \frac{d(S_j(N''))-1}{p_j''} - \frac{1}{p_j'} > \frac{2}{p_j''} - \frac{1}{p_j''} = \frac{1}{p_j''}.\qedhere\]
\end{proof}



By ensuring that $d(S_j(N'))$ is large, we are guaranteed that the regions in which the coefficients are constant are large. Recall that when we add another prime, the coefficients of the new polynomial can be written as signed sums of coefficients of the old polynomial. Therefore, having large regions will allow us to use the maximum coefficient many times in these sums, thereby generating a large coefficient in the new polynomial.

\begin{lemma} \label{central-binom}
Let $N=p_1p_2\cdots p_n$ be generic with $\sum_{i=1}^n \frac{1}{p_i} < 1$. Then there exists $N' = p_1'p_2'\cdots p_n'$ and $q$ prime such that $N'q$ is generic with $\sum_{i=1}^{n} \frac{1}{p_i'}+\frac{1}{q}<1$ such that the height of $P_{N'q}(x)$ is at least $\binom{n-1}{\left\lfloor\frac{n-1}{2}\right\rfloor}$ times the height of $P_N(x)$.
\end{lemma}
\begin{proof}
First find $N'$ and $c$ as in Lemma~\ref{enlarge}, and let $M$ be the coefficient of maximum absolute value in $P_{N'}(x)$. Then, by the Chinese Remainder Theorem, choose $q$ such that $q>N'$ (which implies $\sum \frac{1}{p_i'}+\frac 1q < 1$) and $q^{-1} \equiv \lfloor c p_j' \rfloor \pmod{p_j'}$ for all $j$. Since $cp_j' < d(S_j(N'))$, $N'q$ is generic, and we can find an integer $\bar{k}$ such that the coefficients of $P_{N'}(x)$ corresponding to exponents $\mo{q^{-1}(\bar{k}- \sum_{j \in T} N_j')}{N'}$ are all equal to $M$. Since $M\neq 0$, by the converse to Proposition~\ref{zero},
\[\mo{q^{-1}(\bar{k}- \sum_{j \in T} N_j')}{N'} = \mo{q^{-1}\bar{k}}{N'} - \sum_{j \in T}(\lfloor cp_j' \rfloor  N_j').\]

Observe that
\[cN'-N_j' = (cp_j'-1)N_j' < \lfloor cp_j'\rfloor N_j' \leq cp_j' N_j' = cN'.\]
Let $s = \lfloor\frac{n}{2}\rfloor$, and suppose $T_1, T_2 \subset [n]$ with $|T_1| \leq s$ and $|T_2| > s$. Without loss of generality, assume $p_1'$ is the smallest $p_j'$. Then, since $cp_1'>s+1$,
\[
\sum_{j \in T_2}(\lfloor cp_j' \rfloor  N_j') > (s+1)(cN'-N_1') = csN'+(cp_1'-s-1)N_1' > csN' \geq \sum_{j \in T_1}(\lfloor cp_j' \rfloor  N_j')
\]
Thus all the exponents corresponding to sets $T_1$ of size at most $s$ are bigger than those corresponding to sets $T_2$ of size at least $s+1$. 
Choose $k \equiv \bar{k} \pmod {N'}$ such that
\[
\mo{q^{-1}(\bar{k}- \sum_{j \in T_1} N_j')}{N'} > kq^{-1} > \mo{q^{-1}(\bar{k}- \sum_{j \in T_2} N_j')}{N'}
\]
for all $T_1$ and $T_2$. (Such a $k$ exists because $q^{-1}N'<1$.) But now by Proposition~\ref{truncation},
\begin{align*}
a_{qN'}(k) &= \sum_{T \subset [n]} (-1)^{|T|} a_{N'}(\mo{q^{-1}(\bar{k}- \sum_{j \in T} N_j')}{N'}) \cdot \{\mo{q^{-1}(\bar{k}- \sum_{j \in T} N_j')}{N'}  \leq kq^{-1}\}\\
&= \sum_{\substack{T \subset [n]\\|T| > s}} (-1)^{|T|} M
= \sum_{i=s+1}^{n} (-1)^i \binom{n}{i} M
= (-1)^{s+1} \binom{n-1}{n-s-1} M.
\end{align*}
Thus we have found the desired coefficient with large absolute value.
\end{proof}

By iterating this procedure, we can construct $P_N(x)$ with large height. We obtain the following result.

%

\begin{thm}
$M(n) = 2^{\frac 12 n^2 + O(n\log n)}$.
\end{thm}
\begin{proof}
By Lemma~\ref{central-binom}, we find that $M(n)$ is bounded below by the product of the first $n-2$ central binomial coefficients. By Stirling's formula,
\[\log \binom{i}{i/2} = i \log i - i - 2\left(\frac{i}{2}\log \frac{i}{2}-\frac{i}{2}\right)+O(\log i) = i \log 2 + O(\log i).\]
Summing over the first $n-2$ values of $i$ gives $\log 2 \cdot \frac{n^2}{2} + O(n \log n)$, so exponentiating gives the desired lower bound. Combining this with the upper bound from Proposition~\ref{general-bound} gives the result.
\end{proof}

\section{$P_N(x)$ with height 1}

While we have shown that the maximum height of a polynomial $P_N(x)$ with $n$ distinct prime factors grows exponentially in $n^2$, we will now show that the minimum height of such a polynomial is in fact 1. We first describe how to construct such an $N$.

Recall that if $N=p_1\cdots p_n$, then $S_j(N) = \{\mo{\sum_{i\in T}p_i^{-1}}{p_j} \mid T \subset [n]\backslash\{j\}\}$, and $d(S_j(N))$ is the smallest difference between two distinct elements of $S_j(N)\cap \{p_j\}$.

\begin{lemma} \label{pickprimes}
There exist primes $p_1<p_2<\dots<p_n$ satisfying the following conditions for $1\leq u < v \leq n$:
\begin{enumerate}
\item[(a)] $\mo{p_v^{-1}}{p_u} < d(S_u(p_1p_2\cdots p_{v-1}))$;
\item[(b)] $p_v-\mo{p_u^{-1}}{p_v} < d(S_v(p_1p_2\cdots p_{u-1}p_v))$;
\item[(c)] $\sum_{i=1}^n p_i^{-1} < 1$.
\end{enumerate}
\end{lemma}

Note that conditions (a) and (b) are equivalent to specifying a particular ordering of $S_i(p_1p_2\dots p_n)$ for each $i$.

\begin{proof}
For $n=2$, take $p_1<p_2$ to be any distinct primes.

Suppose that we have constructed $p_1, \dots, p_{n-1}$ satisfying the given conditions. As in Lemma~\ref{enlarge}, we can, in order, increase each $p_i$ by a multiple of the others so that the orders of all the $S_i(p_1p_2\dots p_{n-1})$ are preserved, $d(S_i(p_1p_2\dots p_{n-1})) > 1$ for all $i$, and $p_{i+1}>2p_i$ for $1 \leq i \leq n-2$.

Now let $p_n \equiv 1 \pmod {p_1p_2\dots p_{n-1}}$ be a prime large enough to satisfy (c). We need only check conditions (a) and (b) when $v=n$. Condition (a) is obviously satisfied. For condition (b), note that since $p_i\mo{p_i^{-1}}{p_n} + p_n\mo{p_n^{-1}}{p_i} = p_ip_n+1$,
\[p_n-\mo{p_i^{-1}}{p_n} = \frac{p_n\mo{p_n^{-1}}{p_i}-1}{p_i} = \frac{p_n-1}{p_i}=\lfloor p_i^{-1}p_n\rfloor.\]
Then $d(S_n(p_1p_n))=\min\{\lfloor p_1^{-1}p_n\rfloor, \lfloor (1-p_1^{-1})p_n\rfloor\} = \lfloor p_1^{-1}p_n\rfloor$. Since $\lfloor p_2^{-1}p_n \rfloor \leq \lfloor \frac12 p_1^{-1}p_n\rfloor \leq \frac12 \lfloor p_1^{-1}p_n \rfloor$, $d(S_n(p_1p_2p_n)) = \lfloor p_2^{-1}p_n \rfloor$. Continuing in this manner, we see that $d(S_n(p_1p_2\dots p_ip_n)) = \lfloor p_i^{-1}p_n \rfloor < \lfloor p_{i-1}^{-1}p_n \rfloor = d(S_n(p_1p_2\dots p_{i-1}p_n))$ for all $i$, as desired.
\end{proof}

We claim that if we let $N=p_1p_2\cdots p_n$, then the height of $P_N(x)$ will be 1. The following lemma will allow us to work with the ordering of the $S_i(p_1p_2\dots p_n)$ more explicitly.

Let $h_i^N(k)= h_i(k) = \mo{kN_i^{-1}}{p_i}$, and let $x_i^N(k)=x_i(k)$ be the $i$th coordinate of the region containing $k$. In other words, $x_i(k)+1$ is the number of elements of $S_i(N)$ that are at most $h_i(k)$, so that $x_i(k)$ ranges from 0 to $2^n-1$. (We will sometimes think of $h_i(k)$ as a residue module $p_i$.)

For subsets $V, W \subset [n]\backslash i$, let us write $V \prec_i h \prec_i W$ (and $V \prec_i W$) if \[\mo{\textstyle\sum_{j \in V} p_j^{-1}}{p_i} \leq h < \mo{\textstyle\sum_{j \in W} p_j^{-1}}{p_i}^+.\] Then let  $V^N_i(k) = V_i(k)$ and $W^N_i(k)=W_i(k)$ be the subsets of $[n]\backslash\{i\}$ such that the equations $V_i(k) \prec_i h_i(k) \prec_i W_i(k)$ define the region containing $k$.

\begin{lemma}\label{mod4}
Let $N=p_1p_2\dots p_n$ be defined as in Lemma~\ref{pickprimes}. For $i=n$, $n-1$, or $n-2$, the set $V_i(k) \cap \{n-2, n-1, n\}$ is determined by the residue of $x_i(k)$ modulo 4. Specifically, suppose $V$ and $W$ are subsets of $[n-3]$ such that $V \prec_i W$.
\begin{enumerate}
\item[(i)] If $i=n-2$, then \[V \prec_{n-2} V\cup \{n\} \prec_{n-2} V\cup\{n-1\} \prec_{n-2} V\cup\{n-1,n\} \prec_{n-2} W.\]
\item[(ii)] If $i=n-1$, then
\[V \prec_{n-1} V\cup \{n\} \prec_{n-1} W\cup\{n-2\} \prec_{n-1} W\cup\{n-2,n\} \prec_{n-1} W.\]
\item[(iii)] If $i=n$, then
\[V \prec_{n} W\cup \{n-2,n-1\} \prec_{n} W\cup\{n-2\} \prec_{n} W\cup\{n-1\} \prec_{n} W.\]
\end{enumerate}
\end{lemma}

For instance, if $x_i(k)\equiv 0 \pmod 4$ for $i=n$, $n-1$, or $n-2$, then $V_i(k) \cap \{n-2, n-1, n\} = \varnothing$.

\begin{proof}
This follows easily from Lemma~\ref{pickprimes}: (i) follows from condition (a) when $(u,v) = (n-2, n-1)$ and $(n-2, n)$, (ii) follows from (b) when $(u,v) = (n-2, n-1)$ and (a) when $(u,v)=(n-1,n)$, and (iii) follows from (b) when $(u,v) = (n-2, n)$ and $(n-1, n)$.
\end{proof}

We will also need the following result in the style of Proposition~\ref{pn}.

\begin{prop} \label{two}
Let $p_i$ and $p_j$ be distinct primes dividing $N$. Then modulo $x^N-1$,
\[P_N(x) \equiv \frac{1-x^{\mo{p_j^{-1}}{p_i} N_i}}{1-x^{N_i}}P_{N_i}(x^{p_i}) \cdot \prod_{k \neq i, j} (1-x^{N_{ik}}) + 
x^{\mo{p_j^{-1}}{p_i}N_i}\cdot \frac{x^N-x^{\mo{p_i^{-1}}{p_j}  N_j}}{1-x^{N_j}}P_{N_j}(x^{p_j}) \cdot \prod_{k \neq i, j} (1-x^{N_{jk}}).\]
\end{prop}
\begin{proof}
The right side can be factored as
\[\frac{P_N(x)}{1-x^{N_{ij}}}((1-x^{\mo{p_j^{-1}}{p_i} N_i})+x^{\mo{p_j^{-1}}{p_i}N_i}\cdot (x^N-x^{\mo{p_i^{-1}}{p_j}  N_j})).\]
Subtracting this from the left side and multiplying by $1-x^{N_{ij}}$, we therefore need that
\[P_N(x) \cdot \left(1-x^{N_{ij}}-(1-x^{\mo{p_j^{-1}}{p_i} N_i})-x^{\mo{p_j^{-1}}{p_i}N_i}\cdot (x^N-x^{\mo{p_i^{-1}}{p_j}  N_j})\right)\]
is divisible by $(1-x^N)(1-x^{N_{ij}})$.
But since $\mo{p_j^{-1}}{p_i}N_i+\mo{p_i^{-1}}{p_j}  N_j = N_{ij}(\mo{p_j^{-1}}{p_i}p_j+\mo{p_i^{-1}}{p_j}p_i) = N+N_{ij}$,
\begin{align*}
(1-x^{\mo{p_j^{-1}}{p_i} N_i})+x^{\mo{p_j^{-1}}{p_i}N_i}\cdot (x^N-x^{\mo{p_i^{-1}}{p_j}  N_j}) &= x^{\mo{p_j^{-1}}{p_i} N_i}(x^N-1) + 1-x^{N+N_{ij}}\\
&=(1-x^{N_{ij}})+ (x^N-1)(x^{\mo{p_j^{-1}}{p_i}N_i}-x^{N_{ij}}).
\end{align*}
Thus we just need that $P_N(x) \cdot (1-x^N)(x^{\mo{p_j^{-1}}{p_i}N_i}-x^{N_{ij}})$ is divisible by $(1-x^N)(1-x^{N_{ij}})$, which is clear because both exponents in the final term are divisible by $N_{ij}$.
\end{proof}

Let $0 \leq \bar{k}_i < N_i$ be the integer such that \[p_i\bar{k}_i+\textstyle\sum_{j \in V_i(k)} N_{ij} \equiv k \pmod {N_i}.\]
Observe that $V_j^{N_i}(\bar{k}_i)$ and $W_j^{N_i}(\bar{k}_i)$ are, as in the definition of $V_{j}(k)$ and $W_{j}(k)$, the subsets defining the interval containing $h_j(k)-\mo{p_i^{-1}}{p_j}\cdot \{j \in V_{i}(k)\}$ except that we only consider subsets not containing $i$. (Here, $\{j \in V_i(k)\}$ equals 1 if $j \in V_i(k)$ and 0 otherwise.)

\begin{lemma} \label{shrink}
Let $N=p_1p_2\dots p_n$ be as constructed in Lemma~\ref{pickprimes}. Then for $0 \leq k < N$,
\[a_N(k) = (-1)^{|V_{n-1}(k)|} a_{N_{n-1}}(\bar{k}_{n-1}) \cdot \{n \not \in V_{n-1}(k)\} +
(-1)^{|V_{n}(k)|} a_{N_{n}}(\bar{k}_n) \cdot \{n-1 \in V_n(k)\}.\]
\end{lemma}

\begin{proof}
Consider the right side of Proposition~\ref{two}. For any $0 \leq \bar{k}_{i} < N_i$, the product
\[\frac{1-x^{\mo{p_j^{-1}}{p_i} N_i}}{1-x^{N_i}}P_{N_i}(x^{p_i})=(1+x^{N_i}+x^{2N_i}+\dots + x^{(\mo{p_j^{-1}}{p_i}-1)N_i})P_{N_i}(x^{p_i})\]
contributes $a_{N_{i}}(\bar{k}_{i})$ to the coefficient of $a_N(k)$ for $k=p_i\bar{k}_i+cN_i$ for $0 \leq c < \mo{p_j^{-1}}{p_i}$. Considering all possible $\bar{k}_i$, the resulting values of $k$ are exactly those for which
$0 \leq h_i(k) < \mo{p_j^{-1}}{p_i}$.
Then the entire first term on the right side of Proposition~\ref{two} contributes, for each subset $V \subset [n]\backslash \{i, j\}$,  $(-1)^{|V|}a_{N_i}(\bar{k}_i)$ to each $a_N(k)$ for which $h_i(k)$ lies in the half-open cyclic interval from 
$\mo{\textstyle\sum_{s \in V}p_s^{-1}}{p_i}$ to $\mo{\textstyle\sum_{s \in V \cup \{j\}}p_s^{-1}}{p_i}$.
Now let $i=n-1$ and $j=n$. By Lemma~\ref{mod4}(ii), the $a_N(k)$ that receive a contribution from this first term are those for which $n \not\in V_{n-1}(k)$, and there can only be one such contribution, namely the one from $V=V_{n-1}(k)$. This yields the first term on the right side of the lemma statement.

Similarly, the product
\[x^{\mo{p_j^{-1}}{p_i}N_i}\cdot \frac{x^N-x^{\mo{p_i^{-1}}{p_j}  N_j}}{1-x^{N_j}}P_{N_j}(x^{p_j})\]
contributes, for each $0 \leq \bar{k}_j < N_j$, $-a_{N_j}(\bar{k})$ to those $a_N(k)$ for which \[k\equiv \mo{{p_j}^{-1}}{p_{i}}N_{i}+\mo{p_{i}^{-1}}{p_j}N_j+cN_j+p_j\bar{k}_j=p_j\bar{k}_j+N+N_{ij}+cN_j \equiv p_j\bar{k}_j+N_{ij}+cN_j \pmod{N}\]
for $0 \leq c < p_j-\mo{{p_i}^{-1}}{p_j}$. Over all $\bar{k}_j$, these are those $k$ for which $\mo{p_{i}^{-1}}{p_j}\leq h_j(k) <p_j$. Then the entire second term on the right side of Proposition~\ref{two} contributes, for each subset $W \subset [n]\backslash \{i,j\}$, $(-1)^{|W|+1}a_{N_j}(\bar{k}_j)$ to each $a_N(k)$ for which $h_j(k)$ lies in the half-open cyclic interval from $\mo{\textstyle\sum_{s \in W \cup \{i\}}p_s^{-1}}{p_j}$ to $\mo{\textstyle\sum_{s \in W}p_s^{-1}}{p_j}$. When $i=n-1$ and $j=n$, by Lemma~\ref{mod4}(iii), there can again be at most one contribution to any $a_N(k)$, namely when $n-1 \in V_n(k)$ and $V_n(k)=W\cup \{n-1\}$, which yields the second term above, completing the proof.
\end{proof}

Note that this lemma implies that if $n \in V_{n-1}(k)$ and $n-1 \not \in V_n(k)$, then $a_N(k)=0$.

We can also prove a slightly different version of Lemma~\ref{shrink}. 

\begin{lemma} \label{shrink2}
Let $N=p_1p_2\dots p_n$ be as constructed in Lemma~\ref{pickprimes}. For $0 \leq k < N$, 
\begin{multline*}
a_N(k) = (-1)^{|V_{n-1}(k)|} a_{N_{n-1}}(\mo{\bar{k}_{n-1}-N_{n-1,n}}{N_{n-1}}) \cdot \{n \not \in V_{n-1}(k)\} \\+
(-1)^{|V_n(k)|} a_{N_{n}}(\mo{\bar{k}_n+N_{n-1,n}}{N_n}) \cdot \{n-1 \in V_n(k)\}.\end{multline*}
\end{lemma}
\begin{proof}
By a similar argument as in Proposition~\ref{two},
\[P_N(x) \equiv \frac{x^N-x^{\mo{p_j^{-1}}{p_i} N_i}}{1-x^{N_i}}P_{N_i}(x^{p_i}) \cdot \prod_{k \neq i, j} (1-x^{N_{ik}}) + 
x^{\mo{p_j^{-1}}{p_i}N_i}\cdot \frac{1-x^{\mo{p_i^{-1}}{p_j}  N_j}}{1-x^{N_j}}P_{N_j}(x^{p_j}) \cdot \prod_{k \neq i, j} (1-x^{N_{jk}}).\]
modulo $x^N-1$. Now apply the same argument as in Lemma~\ref{shrink} with $i=n$ and $j=n-1$.
\end{proof}
(Alternatively, one can use Proposition~\ref{minus-n} to compare the corresponding terms in Lemma~\ref{shrink} and Lemma~\ref{shrink2}.)

We are now ready to prove the main theorem of this section.

\begin{thm}
Let $N=p_1p_2\dots p_n$ be as constructed in Lemma~\ref{pickprimes}. Then for all $k$, $|a_N(k)| \leq 1$.
\end{thm}
\begin{proof}
We proceed by induction on $n$, having proven the cases $n \leq 3$ previously. Note that if $p_1, p_2, \dots, p_n$ satisfy the conditions of Lemma~\ref{pickprimes}, then so does any subset of them.

By the induction hypothesis, we may assume that none of the four terms on the right sides of Lemma~\ref{shrink} and Lemma~\ref{shrink2} vanish. In particular, this implies that $n \not \in V_{n-1}(k)$, so $x_{n-1}(k)$ is even, and similarly $x_n(k)$ is odd.

Suppose $x_n(k) \equiv 1 \pmod 4$, so that $n-2, n-1 \in V_{n}(k)$. We claim that $x_{n-1}(k) \equiv 2 \pmod 4$ and $x_{n-2}(k) \equiv 2 \pmod 4$. Indeed, suppose $x_{n-1}(k) \equiv 0 \pmod 4$, so that $n-2, n \not \in V_{n-1}(k)$.

If $x_{n-2}(k) \equiv 0 \pmod 4$, then let $m' = \mo{\bar{k}_n+N_{n-1,n}}{N_n}$. Since $n-2 \in V_{n}(k)$ and $W_{n-2}(k)$ contains $n$ but not $n-1$, it follows that $W^{N_n}_{n-2}(\bar{k}_n)\backslash \{n\}=W_{n-2}(k)$ does not contain $n-1$, so $n-1 \in V^{N_n}_{n-2}(\bar{k}_n)=V^{N_n}_{n-2}(m')$. Similarly since $n-1 \in V_n(k)$ and $W_{n-1}(k)$ contains $n$ but not $n-2$, we have that $n-2 \in V^{N_n}_{n-1}(\bar{k}_n)$. Then since $x_{n-1}^{N_n}(m') = x_{n-1}^{N_n}(\bar{k}_n)+1$, $n-2 \not\in V^{N_n}_{n-1}(m)$. But then applying Lemma~\ref{shrink} to $a_{N_n}(m')$ implies that it equals 0, which we assumed was not the case.

Likewise, if $x_{n-2}(k) \not \equiv 0 \pmod 4$, then let $\ell = \bar{k}_{n-1}$. A similar argument to above implies that $n-2 \not \in V_{n}^{N_{n-1}}(\ell)$ and $n \in V_{n-2}^{N_{n-1}}(\ell)$, so that again by Lemma~\ref{shrink}, $a_{N_{n-1}}(\ell) = 0$, which we assumed was not the case. It follows that if $x_n(k) \equiv 1 \pmod 4$, then $x_{n-1}(k) \equiv 2 \pmod 4$. Then we still have that $n-2 \not \in V_n^{N_{n-1}}(\ell)$, but now $n \in V_{n-2}^{N_{n-1}}(\ell)$ if and only if $x_{n-2}(k) \not \equiv 2 \pmod 4$. Therefore if $a_{N_{n-1}}(\ell) \neq 0$, we must have $x_{n-2}(k) \equiv 2 \pmod 4$.

Assume then that $x_n(k) \equiv 1 \pmod 4$ and $x_{n-1}(k) \equiv x_{n-2}(k) \equiv 2\pmod 4$, and let $m=\bar{k}_n$ and $\ell=\bar{k}_{n-1}$ so that $a_N(k) = (-1)^{|V_{n-1}(k)|}a_{N_{n-1}}(\ell) + (-1)^{|V_n(k)|}a_{N_n}(m)$. From above, $n-2 \not \in V^{N_{n-1}}_{n}(\ell)$ and $n \not \in V_{n-2}^{N_{n-1}}(\ell)$, so 
\[a_{N_{n-1}}(\ell) = (-1)^{|V^{N_{n-1}}_{n-2}(\ell)|}a_{N_{n-1, n-2}}(\bar{\ell}_{n-2}),\]
where $\bar\ell_{n-2}$ is defined as in Lemma~\ref{shrink} applied to $a_{N_{n-1}}(\ell)$. 
Similarly, $n-2 \not \in V^{N_{n}}_{n-1}(m)$ and $n-1 \not \in V^{N_n}_{n-2}(m)$, so
\[a_{N_{n}}(m) = (-1)^{|V^{N_n}_{n-2}(m)|}a_{N_{n, n-2}}(\bar{m}_{n-2}).\]
Since $V_{n-2}^{N_{n-1}}(\ell) = V_{n-2}^{N_n}(m) = V_{n-2}(k) \cap [n-3]$, we have that
\[\pm a_N(k) = (-1)^{|V_{n-1}(k)|}a_{N_{n-1, n-2}}(\bar{\ell}_{n-2}) + (-1)^{|V_n(k)|}a_{N_{n, n-2}}(\bar{m}_{n-2}). \tag{$*$}\]

Let $s = \mo{\bar{k}_{n-2} - N_{n, n-2}}{N_{n-2}}$. We claim that the right side of $(*)$ equals $-a_{N_{n-2}}(s)$, which will complete the proof in this case by the inductive hypothesis.

It is straightforward to check that $V_n^{N_{n-2}}(s) = V_n(k)\backslash\{n-2\}$ and $V_{n-1}^{N_{n-2}}(s) = V_{n-1}(k)\backslash\{n-2\}$. Then by Lemma~\ref{shrink}, 
\begin{align*}
a_{N_{n-2}}(s) &= (-1)^{|V_{n-1}^{N_{n-2}}(s)|}a_{N_{n-1, n-2}}(\bar{s}_{n-1})+(-1)^{|V_{n}^{N_{n-2}}(s)|}a_{N_{n, n-2}}(\bar{s}_{n})\\
&=-\left((-1)^{|V_{n-1}(k)|}a_{N_{n-1,n-2}}(\bar{s}_{n-1})+ (-1)^{|V_{n}(k)|}a_{N_{n, n-2}}(\bar{s}_{n})\right)
\end{align*}
Therefore it suffices to show that $a_{N_{n-1,n-2}}(\bar{\ell}_{n-2}) = a_{N_{n-1, n-2}}(\bar{s}_{n-1})$ and $a_{N_{n,n-2}}(\bar{m}_{n-2}) = a_{N_{n, n-2}}(\bar{s}_{n})$.

For all $i \in [n]\backslash \{n-1, n-2\}$,
\[
h_i^{N_{n-1,n-2}}(\bar{s}_{n-1})=h_i^{N_{n-2}}(s) - \mo{p_{n-1}^{-1}}{p_i}\cdot \{i \in V_{n-1}^{N_{n-2}}(s)\}.
\]
When also $i \neq n$, $h_i^{N_{n-2}}(s)=h_i^{N_{n-2}}(\bar{k}_{n-2})$, so since $V_{n-2}^{N_{n-1}}(\ell) = V_{n-2}(k)\backslash\{n-1\}$,
\begin{align*}
h_i^{N_{n-1,n-2}}(\bar{s}_{n-1}) &= h_i^{N_{n-2}}(\bar{k}_{n-2}) - \mo{p_{n-1}^{-1}}{p_i}\cdot \{i \in V_{n-1}^{N_{n-2}}(s)\}\\
&=h_i^N(k)-\mo{p_{n-2}^{-1}}{p_i}\cdot \{i \in V_{n-2}(k)\}-\mo{p_{n-1}^{-1}}{p_i}\cdot \{i \in V_{n-1}^{N_{n-2}}(s)\}\\
&=h_i^N(k)-\mo{p_{n-1}^{-1}}{p_i}\cdot \{i \in V_{n-1}(k)\}-\mo{p_{n-2}^{-1}}{p_i}\cdot \{i \in V_{n-2}^{N_{n-1}}(\ell)\}\\
&=h_i^{N_{n-1}}(\ell)-\mo{p_{n-2}^{-1}}{p_i}\cdot \{i \in V_{n-2}^{N_{n-1}}(\ell)\}\\
&=h_i^{N_{n-1,n-2}}(\bar{\ell}_{n-2}).
\end{align*}
When $i=n$,
\begin{multline*}h_n^{N_{n-1,n-2}}(\bar{s}_{n-1}) = h_n^{N_{n-2}}(s) = h_n^{N_{n-2}}(\bar{k}_{n-2})-\mo{p_{n-2}^{-1}}{p_n}=h_n^N(k)-\mo{p_{n-2}^{-1}}{p_n}\\=h_n^{N_{n-1}}(\ell)-\mo{p_{n-2}^{-1}}{p_n}=h_n^{N_{n-1,n-2}}(\bar{\ell}_{n-2})-\mo{p_{n-2}^{-1}}{p_n}.\end{multline*}
But since $x_n(k) \equiv 1 \pmod 4$, Lemma~\ref{mod4} implies that even with the shift of $\mo{p_{n-2}^{-1}}{p_n}$, $\bar{s}_{n-1}$ and $\bar\ell_{n-2}$ still lie in the same region for $N_{n-1, n-2}$, proving that $a_{N_{n-1,n-2}}(\bar{\ell}_{n-2}) = a_{N_{n-1, n-2}}(\bar{s}_{n-1})$. A similar argument proves that \[h_{n-1}^{N_{n, n-2}}(\bar s_n) = h_{n-1}^{N_{n,n-2}}(\bar m_{n-2})-\mo{p_{n-2}^{-1}}{p_{n-1}}=h_{n-1}^N(k)-\mo{p_{n-2}^{-1}}{p_{n-1}}-\mo{p_n^{-1}}{p_{n-1}},\] while $h_i^{N_{n, n-2}}(\bar s_n) = h_i^{N_{n,n-2}}(\bar m_{n-2})$ for $i \in [n-3]$, and Lemma~\ref{mod4} once again shows that $\bar{s}_n$ and $\bar{m}_{n-2}$ lie in the same region for $N_{n, n-2}$. This completes the proof in the case that $x_n(k) \equiv 1 \pmod 4$.


The only case remaining is when $x_n(k) \equiv 3 \pmod 4$. This follows by essentially the same argument as the previous case: first, in order for the four terms on the right sides of Lemma~\ref{shrink} and Lemma~\ref{shrink2} not to vanish, we must have $x_{n-1}(k) \equiv x_{n-2}(k) \equiv 0 \pmod 4$. Then let $\ell' = \mo{\bar{k}_{n-1}-N_{n-1,n}}{N_{n-1}}$ and $m'=\mo{\bar{k}_n+N_{n-1,n}}{N_n}$ as in Lemma~\ref{shrink2}. Once again, we find that $n-2 \not\in V_n^{N_{n-1}}(\ell')=V_n(k) \backslash \{n-1\}$ and $n \not\in V_{n-2}^{N_{n-1}}(\ell') = V_{n-2}(k)$, so that by Lemma~\ref{shrink},
\[a_{N_{n-1}}(\ell') = (-1)^{|V_{n-2}^{N_{n-1}}(\ell')|}a_{N_{n-1,n-2}}(\bar\ell'_{n-2}).\]
Similarly
\[a_{N_n}(m') = (-1)^{|V_{n-2}^{N_{n-2}}(m)|}a_{N_{n,n-2}}(\bar m'_{n-2}).\]
Then, since $V_{n-2}^{N_{n-1}}(\ell')=V_{n-2}^{N_{n}}(m')=V_{n-2}(k)$,
\begin{align*}
a_N(k) &= (-1)^{|V_{n-1}(k)|}a_{N_{n-1}}(\ell')+(-1)^{|V_n(k)|}a_{N_n}(m')\\
&= \pm\left((-1)^{|V_{n-1}(k)|}a_{N_{n-1,n-2}}(\bar\ell'_{n-2})+(-1)^{|V_n(k)|}a_{N_{n,n-2}}(\bar m'_{n-2})\right) \tag{$**$}
\end{align*}

An analogous argument to the previous case now shows that the two terms on the right side of $(**)$ equal the two terms on the right side of Lemma~\ref{shrink2} when applied to $a_{N_{n-2}}(\bar{k}_{n-2})$, which completes the proof by induction.
\end{proof}

\section{Acknowledgments}
The author would like to thank Victor Reiner for his useful discussions and overall encouragement. This work was supported by a National Science Foundation Mathematical Sciences Postdoctoral Research Fellowship.

\bibliography{cyc}
\bibliographystyle{plain}

\end{document}